\documentclass[11pt,times,letter]{article}

\usepackage{amsmath, amsfonts,amsthm,amssymb,bm}
\usepackage{dsfont}

\usepackage[english]{babel}

\usepackage[normalem]{ulem}

\usepackage{latexsym,amssymb,amsfonts,graphicx}
\usepackage{amsmath,dsfont}
\usepackage{verbatim}
\usepackage{mathrsfs}
\usepackage{bm}
\usepackage{color}
\usepackage{epsfig}
\usepackage{subfigure}

\usepackage{url}
\usepackage{bm}
\usepackage{natbib}

\usepackage[colorlinks,linkcolor=blue,citecolor=blue,anchorcolor=blue]{hyperref}

\usepackage{tikz}
\usetikzlibrary{arrows.meta, positioning}

\newcommand{\Px}{ \mathbb{P} }

\newcommand{\Ex}{ \mathbb{E} }

\def\esssup_#1{\underset{#1}{\mathrm{ess\,sup\, }}}
\def\essinf_#1{\underset{#1}{\mathrm{ess\,inf\, }}}
\def\argmax_#1{\underset{#1}{\mathrm{arg\,max\, }}}
\def\argmin_#1{\underset{#1}{\mathrm{arg\,min\, }}}

\newcommand{\Fx}{\mathbb{F} }

\newcommand{\F}{\mathcal{F}}

\newcommand{\R}{\mathds{R}}

\newcommand{\C}{{\tt{C}}}

\newtheorem{theorem}{Theorem}[section]
\newtheorem{definition}{Definition}[section]
\numberwithin{equation}{section}

\newtheorem{proposition}[theorem]{Proposition}
\newtheorem{remark}[theorem]{Remark}
\newtheorem{lemma}[theorem]{Lemma}
\newtheorem{corollary}[theorem]{Corollary}

\allowdisplaybreaks[4]

\definecolor{Red}{rgb}{1.00, 0.00, 0.00}

\definecolor{DRed}{rgb}{0.5, 0.00, 0.00}

\definecolor{Blue}{rgb}{0.00, 0.00, 1.00}

\definecolor{Green}{rgb}{0.0, 0.4, 0.0}

\title{ Mean Field Game of Optimal Tracking Portfolio}
\author{Lijun Bo \thanks{Email: lijunbo@ustc.edu.cn, School of Mathematics and Statistics, Xidian University, Xi'an, 710126, China.}
\and
Yijie Huang \thanks{Email: yijie.huang@polyu.edu.hk, Department of Applied Mathematics, The Hong Kong Polytechnic University, Kowloon, Hong Kong, China.}
\and
Xiang Yu \thanks{Email: xiang.yu@polyu.edu.hk, Department of Applied Mathematics, The Hong Kong Polytechnic University, Kowloon, Hong Kong, China.}
}

\date{\vspace{-0.5in}}
\begin{document}
\maketitle

\begin{abstract}
This paper studies the mean field game (MFG) problem arising from a large population competition in fund management, featuring a new type of relative performance via the benchmark tracking.  In the $n$-player model, each agent aims to minimize the expected largest shortfall of the wealth with reference to the benchmark process, which is modeled by a linear combination of the population's average wealth process and a market index process. With a continuum of agents, we formulate the MFG problem with a reflected state process. We establish the existence of the mean field equilibrium (MFE) using the  partial differential equation (PDE) approach. Firstly, by applying the dual transform, the best response control of the representative agent can be characterized in analytical form in terms of a dual reflected diffusion process. As a novel contribution, we verify the consistency condition of the MFE in separated domains with the help of the duality relationship and properties of the dual process. Moreover, based on the MFE, we construct an approximate Nash equilibrium for the $n$-player game when the number $n$ is sufficiently large.     
\end{abstract}
\vspace{0.1in}

\section{Introduction}\label{sec:introduction}
Portfolio  management with benchmark tracking has been an active research topic in quantitative finance, which involves constructing a portfolio that closely tracks a chosen benchmark.  The goal is usually to minimize the tracking error (for instance, the variance or downside variance relative to the index value or return). Some related studies along this direction can be found in \cite{Browne9,Browne99,Browne00}, \cite{Gaivoronski05},  \cite{YaoZZ06}, \cite{Strub18}, \cite{NLFC}, among others.   Recently,  \cite{BoLiaoYu21} proposed a tracking portfolio formulation, termed ``relaxed benchmark tracking", in which an agent is allowed to use (fictitious) capital injection to outperform a benchmark process and aims to minimize the cost of the capital injection. The fictitious injected capital can also be regarded as a risk measure of the expected largest shortfall with reference to the benchmark. Later, this relaxed benchmark tracking formulation has been further generalized to the Merton's consumption problem under different types of benchmark processes in \cite{BHY24a,BHY23,BHY24b} and to the incomplete market model using the continuous-time reinforcement learning approach in \cite{BHY25}.

 In the present paper, we study the relaxed tracking portfolio mean field game (MFG) problem. In particular, we consider a system of $n$ competitive agents who are concerned with their performance relative to the population's average wealth.  When the number of agents goes to infinity, we explore the corresponding MFG problem that is more tractable. Introduced independently by  \cite{Huang06} and \cite{Lasry07}, MFGs have emerged as a significant research area in recent years, serving as a powerful tool for studying  large population games. The existing literature on $n$-player and mean field portfolio games with relative performance mostly consider the utilities  defined by comparing the agent's own wealth to the population's average wealth.  To name a few,
\cite{Espinosa15} established the existence and uniqueness of a Nash equilibrium for an $n$-agent game under general utility functions and portfolio constraints. \cite{Lacker19} discussed both $n$-player and MFG problems of optimal portfolio under competition and constant relative risk aversion (CRRA) or constant absolute risk aversion (CARA) relative performance criteria, deriving explicit solutions. Their work was extended by   \cite{Lacker20}  to include consumption,  and by \cite{FuZhou23} to models with random return rates and volatility. Applying the master equation approach,  \cite{Souganidis24} analyzed a MFG problem with relative performance, deriving the MFG solution and using it to approximate the $n$-player game equilibrium. \cite{BWY24a} examined the equilibrium consumption under linear and multiplicative external habit formation. \cite{BWY24b} considered a relative performance MFG problem with contagious jump risk.

In contrast, we consider a new type of relative performance using the expected largest shortfall with reference to a benchmark that incorporates the population's average wealth and an exogenous market index process. Each agent needs to strategically choose an investment strategy to minimize this new risk measure. See also some related concepts of risk measures such as the conventional expected shortfall with respect to a terminal random variable in Pham \cite{Pham02} and the intra-horizon expected shortfall in Farkas et al. \cite{Farkas21}, and references therein. After some equivalent transformations, we need to solve a type of MFGs with state reflection at zero. The primary interest in  reflected state processes stems from their ability to model hard constraints and boundary behaviors, which are ubiquitous in economic, financial, and physical systems. Unlike soft penalties that can be violated at a cost, reflection imposes a strict, binding constraint that the state cannot cross. Mathematically, this is represented by a reflection term that instantaneously ``pushes'' the state back into the admissible domain whenever it reaches the boundary. Stochastic control of reflected diffusions, stemming from running maximum processes or singular control, has garnered increasing attention in recent research.  To name a few,  \cite{Han16} concerned the optimal pricing barrier control in a regulated financial market system modelled by a reflected diffusion process.  \cite{WeerasingheaZhu16} formulated a Brownian inventory control problem with capacity expansion as an equivalent stochastic control problem with reflection, solvable through analysis of the corresponding Hamilton-Jacobi-Bellman (HJB)  equation.  \cite{Ferrari19} investigated a class of singular control problems for diffusions reflected at zero.  However, the research on MFGs with reflected diffusion process remains relatively unexplored. \cite{Bayraktar2019} formulated a MFG problem with drift-controlled diffusion process with reflections for large symmetric queuing systems in heavy traffic with strategic servers. By using the master equation approach,  \cite{Ricciardi2023} handled the convergence of Nash equilibrium in an $n$-player differential game towards the optimal strategies in the MFG for drift control with reflection. Recently, \cite{BWY25} explored a MFG problem with state-control joint law dependence and the reflected state process, in which the existence of MFEs is established by developing the compactification method and the connections between the MFG and the $n$-player game problems are also discussed.  

In this work, we solve the MFG problem with state reflections by using the analytical approach. In the first step, we fix a deterministic differentiable function that represents the average wealth, and study the optimization problem of the representative agent. We reformulate the control problem with largest shortfall by introducing  an auxiliary controlled state process with reflection. By means of the dual transform, probabilistic representation and stochastic flow analysis, the existence of classical solution to the HJB equation is obtained and the verification theorem on the best response control is proved.  As shown in Theorem \ref{thm:verification}, we emphasize that the optimal portfolio process exhibits path dependence on both the wealth process $V=(V_t)_{t\in[0,T]}$ and the benchmark process $Z=(Z_t)_{t\in[0,T]}$. This inherent path dependence makes the direct analysis based on the original state process $V=(V_t)_{t\in[0,T]}$ analytically intractable. Consequently, the auxiliary reflected state process $X=(X_t)_{t\in[0,T]}$ becomes indispensable, as it facilitates the representation of optimal strategies in feedback form (see Remark \ref{rem:XV}).

Then, we aim to  verify the consistency condition and find an MFE in the second step based on the characterization of the value function and feedback function of the best response control derived in the previous step. It is sufficient to show the existence of the fixed point for the mapping that stems from the consistency condition.  A fundamental challenge is the piecewise expression. The optimal strategy is not in a unified form but it takes different expression across different regions (outperforming vs. underperforming), and the boundaries of the regions also depend on the population's state distribution. This complexity is caused by the interplay of maximum process (largest shortfall), the benchmark tracking, and mean field interactions, transforming the consistency condition into a non-standard fixed point problem. In response, we introduce the reflected dual process, which helps to provide a probability representation of best response portfolio strategy and establish the existence and uniqueness of the fixed point of the dual variable. Accordingly, we return to the primal state variable by using the dual relationship and find the MFE in piecewise form by constructing the fixed point. We split the domain into two regions depending on the model parameters, the initial level of the wealth and benchmark processes and the time horizon: on the outperforming region, we derive the explicit expression of the optimal investment strategy; on the underperforming region, not explicit but calculable representations of the optimal investment strategy is provided. Finally, with the aid of MFE, we construct an approximate Nash equilibrium for the game when the number $n$  is sufficiently large. The flowchart in Figure \ref{fig:chart} outlines the roadmap of our proposed approach, which utilizes the auxiliary dual process.

We summarize the key contributions as threefold. First, we generalize the single-agent framework of \cite{BoLiaoYu21} by formulating a MFG problem that introduces a novel relative performance criterion based on the population’s average wealth process. Second, we solve a new type of MFG problem with state reflections analytically and provide a characterization of one MFE in this new context. Third, we develop a convex dual method based on reflected state processes that enables the rigorous proof of the MFE’s consistency condition through the  analysis of the reflected dual process. We significantly extend the dual method initiated in \cite{BoLiaoYu21} and \cite{BHY24a}, which employed the Legendre-Fenchel transform and stochastic flow analysis to study solution regularity. Subsequent developments in \cite{BHY24b} and \cite{BHY23} introduced reflected dual processes to handle problems with general utility functions and drawdown constraints. In the present work, we advance this methodology further by characterizing the best-response control through a reflected dual process in a mean field game setting.

The remainder of the paper is organized as follows. Section \ref{sec:model} introduces the finite-agent and MFG problems of the optimal tracking portfolio. Section \ref{sec:auxiliary} proposes an auxiliary state process with reflection and derive the associated HJB equation with Neumann boundary condition for the best response control problem, in which the best response portfolio process of the representative agent can be obtained in analytical form in terms of a reflected dual process. Section \ref{sec:MFG}  verifies the consistency condition and establishes the existence of the MFE with the aid of the reflected dual process. Section \ref{sec:numerical} presents some quantitative properties and numerical examples of the MFE. Section \ref{sec:approximate} constructs an approximate Nash equilibrium for the game with a sufficiently large but finite number $n$ of agents. Section \ref{sec:conclusion} summarizes the main results and discusses future research directions. The proofs of main results in previous sections are collected in Section \ref{sec:proofs}.

\begin{figure}
\begin{center}
\begin{tikzpicture}[
    node distance=0.6cm,
    box/.style={draw, rectangle, minimum width=6cm, minimum height=3cm, align=center}
]
    \node[box] (def) {{\bf MFE}:\\$\Ex\big[V_t^{*,f}\big]=\mathrm{v}+\int_0^t f(s)ds$\\[0.4em] (see  Definition \ref{def:MFG})};
    \node[box, right=of def] (xt) {$K_t=\lambda \overline{V}_t+(1-\lambda )Z_t-V_t$\\ $L_t=0\vee \sup_{s\leq t}K_s$\\[0.4em] {\bf Auxiliary process}:\\ $X_t:=L_t-K_t$ \\(see Equation \eqref{state-X})};
    \node[box, below=of xt] (ft) {{\bf Dual process}:\\ $Y_t=u_x(t,X_t^*,Z_t)$\\[0.4em] $dY_t= Y_t(\rho dt-\frac{\mu}{\sigma} dW_t)-dL_t^Y$\\(see Lemma \ref{lem:dual-process})};
    \node[box, left=of ft] (int) {$R_t=-\ln(Y_t)$\\ {\bf Fixed point}:  \\ $f(t)$ for {\bf dual variable} $(r,z)$\\ (see Lemma \ref{lem:fixpoint-r})};

    \draw[-Latex] (def) -- (xt);
    \draw[-Latex] (xt) -- (ft);
    \draw[-Latex] (ft) -- (int);
    \draw[-Latex] (int) -- (def)
     node[midway, right] { Theorem \ref{thm:MF-equilibrium}};
\end{tikzpicture}
\caption{The road map of solving the MFG problem.}\label{fig:chart}
\end{center}
\end{figure}

\section{Finite-Agent and MFG Problem}\label{sec:model}

Let $(\Omega,\F,\Fx,\Px)$ be a filtered probability space with the filtration $\Fx=(\F_t)_{t\in[0,T]}$ satisfying the usual conditions (i.e., $\F_0$ contains all $\Px$-null sets and $\Fx$ is right continuous).  Consider a financial market consisting of $n$ agents and $n$ risky assets, in which the $i$-th agent invests her wealth in the $i$-th risky asset whose price dynamics is given by 
\begin{equation}\label{eq:stock}
  \frac{d S_t^i}{S_{t}^i} = \mu^{i} dt +\sigma^{i} dW_t^i,\quad t\in [0,T].
\end{equation}
Here the return rate $\mu^{i}>0$, the volatility  $\sigma^{i}>0$ and $(W^1,\ldots,W^n)^{\top}=(W_t^1,\ldots,W_t^n)^{\top}_{t\in[0,T]}$ is an $n$-dimensional Brownian motion, specified to each individual risky asset, standing for the idiosyncratic noise. It is assumed that the riskless interest rate $r=0$, which amounts to the change of num\'{e}raire. From this point onwards, all processes including the wealth process and the benchmark process are defined after the change of num\'{e}raire. At time $t\geq0$, let $\theta_t^i$ be the amount of wealth that agent $i$ allocates in the asset $S^i$. Then, the self-financing wealth process of agent $i$ satisfies the following controlled stochastic differential equation (SDE):
 \begin{align}\label{eq:Vt}
dV_t^i=\theta^i_t \mu^{i} dt+\theta^i_t\sigma^{i} dW_t^i,\quad t\in [0,T].
\end{align}

We now introduce the relaxed benchmark tracking problem based on the expected largest shortfall by allowing the wealth process to fall below the given benchmark process. The benchmark process is modeled as a linear combination of the average wealth process $\overline{V}_t^n=\frac{1}{n}\sum_{i=1}^n V^i_t$  and the market index level $Z_t^i$ that agent $i$ is concerned. The interaction of $n$ agents occurs through the average wealth process $\overline{V}^n_t$ which depicts the relative performance concern by agent $i$ with respect to other peers. Herein, the market index process  $Z^i=(Z_t^i)_{t \in [0,T]}$ for $i=1,\ldots,n$ is described as the following geometric Brownian motion (GBM) that $Z_0^i=z\in\R_+:=[0,\infty)$ and
\begin{align}\label{eq:Zt-n}
d Z_t^i=\mu_Z^{i} Z_t^i d t+\sigma_Z^{i} Z_t^i d W_t^{i},
\end{align}
where the return rate $\mu_Z^{i}>0$ and the volatility $\sigma_Z^{i}>0$ satisfying $\frac{\mu_Z^{i}}{\sigma_Z^{i}}>\frac{\mu^{i}}{\sigma^{i}}$. This condition ensures that the market index’s Sharpe ratio is sufficiently high. Economically, this means that the fund manager in our model restricts attention to benchmark processes that exhibit strong market performance. In practice, fund managers naturally prefer benchmarks with high Sharpe ratios, making this assumption economically plausible.

Given the benchmark process $Z^{i}=(Z_t^i)_{t\in[0,T]}$, denote by $A^i=(A_t^i)_{t\in[0,T]}$ the largest shortfall (up to time $t$) when the wealth falls below the benchmark, that is
\begin{align*}
A^{i}_t =0\vee \sup_{s\leq t}\left(\lambda^i \overline{V}_s +(1-\lambda^i)Z_s^i-V_{s}^i\right), ~~\forall t\in [0,T],
\end{align*}
where $\lambda^{i} \in[0,1]$  represents the competition weight of agent $i$ towards her relative performance concern. Then,  $\Ex[A_0^i+\int_0^{T} e^{-\rho t}dA^i_t]$  defines a risk measure on the expected largest shortfall with reference to the benchmark. The agent needs to strategically choose the portfolio control to minimize this risk measure on tracking.

\begin{remark}
 Formally, the model parameters of agent $i$ in the $n$-agent game are given by $(\mu^{i,n}, \sigma^{i,n}, \mu_Z^{i,n}, \sigma_Z^{i,n}, \lambda^{i,n})$. To simplify notation, we will denote these parameters as $(\mu^{i}, \sigma^{i}, \mu_Z^{i}, \sigma_Z^{i}, \lambda^{i})$, omitting the dependence on $n$.
\end{remark}

Our choice of a linear combination of average population wealth and a market index is motivated by its economic interpretation as capturing a relative performance concern that is both social (peer-driven) and market-based. Specifically, the average population wealth component $\lambda^{i}{\overline{V}}_t^n$ captures the agent’s preference to outperform their peer group. For instance,  \cite{Espinosa15} examined optimal investment under relative wealth concerns, where agents maximize a convex combination of their own wealth and its deviation from the average wealth of peers. The market index component $(1-\lambda^{i})Z_t^i$ such as the S\&P 500, NASDAQ, or Dow Jones represents a standard financial benchmark and a common performance target for fund managers. By integrating both components above, our benchmark enables agents to evaluate performance against both direct competitors and the market--a structure aligning with the practical behavior.

Mathematically speaking, for $i=1,\ldots,n$, agent $i$ now aims to minimize the expected largest shortfall at the terminal time $T$, that is, for $(\mathrm{v},z)\in\R_+^2$,
\begin{align}\label{eq:n-player}
 {\rm w}^i(\mathrm{v},z):=\text{$\inf_{\theta^i\in\mathbb{U}} \Ex\left[ A_0^i+\int_0^{T} e^{-\rho t}dA^i_t\Big|V^i_0=\mathrm{v},Z_0^i=z\right]$}
\end{align}
where $\rho>0$ is the discount factor. The admissible control set $\mathbb{U}$ is defined as the set of processes $\theta=(\theta_t)_{t\in[0,T]}$ such that $\theta$ is an $\Fx$-adapted real-valued process satisfying $\Ex[\int_0^T |\theta_t|^2dt]<\infty$. Here, for technical convenience, we assume that the initial wealth of each agent and the initial level of market index process are the same constants $(\mathrm{v},z)\in \R_+^2$.

In the setting of a continuum of agents, we formulate the MFG problem in the limit as $n\rightarrow\infty$. To do it, denote by $\overline{V}=(\overline{V}_t)_{t\in[0,T]}$ the average wealth process of the mean field population.  We assume that, as $n \to \infty$, $(\mu^{i}, \sigma^{i}, \mu_Z^{i}, \sigma_Z^{i}, \lambda^{i})$ converge uniformly to a common limit $(\mu, \sigma, \mu_Z, \sigma_Z, \lambda)\in(0,\infty)^4\times[0,1]$ satisfying  $\frac{\mu_Z}{\sigma_Z}>\frac{\mu}{\sigma}$. When $n\rightarrow\infty$, we expect that $t\mapsto\overline{V}_t$ is a deterministic differentiable function with the form $d\overline{V}_t=f(t)dt$ for some positive { $f\in\C([0,T])$ and $\overline{V}_0=\mathrm{v}\in\R_+$. Here,  $\C([0,T])$ denotes the set of real-valued continuous functions on $[0,T]$. Then, our MFG problem is to find $(\theta^*,A^*)=(\theta_t^*,A_t^*)_{t\in[0,T]}\in\mathbb{U}$ such that the best response satisfies the so-called consistency condition. 

For a given function $t\mapsto f(t)$ (or $t\mapsto\overline{V}_t$), we first consider the stochastic control problem for a representative agent that  
\begin{align}\label{eq:mfg}
 {\rm w}(\mathrm{v},z):=\inf_{\theta\in \mathbb{U}} \Ex\left[ A_0+ \int_0^{T} e^{-\rho t}dA_t\Big|V_0=\mathrm{v},Z_0=z \right].
\end{align}
Here the wealth process $V=(V_t)_{t\in[0,T]}$ of the representative agent satisfies that $V_0=\mathrm{v}\in\R_+$ and 
\begin{align}\label{eq:Vt-MFG}
 dV_t=\theta_t \mu dt+\theta_t\sigma dW_t,~~t\in (0,T].
\end{align}
The market index process $Z=(Z_t)_{t \in [0,T]}$ of the representative agent is given by $Z_0=z\in\R_+$ and 
\begin{align}\label{eq:Zt-MFG}
d Z_t=\mu_Z Z_t d t+\sigma_Z Z_t d W_t,~~t\in(0,T].
\end{align}
The largest shortfall process $A=(A_t)_{t\in[0,T]}$ of the representative agent  is given by
\begin{align*}
A_t =0\vee \sup_{s\leq t}\left(\lambda f(s) +(1-\lambda)Z_s-V_{s}\right), \quad \forall t\in [0,T],
\end{align*}

We next give the definition of the MFE for our MFG problem:
\begin{definition}[MFE]\label{def:MFG}
{\it For a given function $f\in\C([0,T])$, let $\theta^{*,f}=(\theta_t^{*,f})_{t\in[0,T]}\in \mathbb{U}$ be the best response solution to the stochastic control problem \eqref{eq:mfg} for the representative agent. The triplet $(f,\theta^{*,f})$ is called a MFE if  the strategy $\theta^{*,f}$ is the best response to  \eqref{eq:mfg} such that $\Ex[V_t^{*,f}]=\mathrm{v}+\int_0^t f(s)ds$ for all $t\in[0,T]$, where $V^{*,f}=(V_t^{*,f})_{t\in[0,T]}$ is the wealth process \eqref{eq:Vt-MFG} under the portfolio strategy $\theta^{*,f}$.} 
\end{definition}

Based on the above definition, finding an MFE is to solve the following two-step problem:
\begin{itemize} 
\item {\bf Step 1.} For a given deterministic positive function $f\in\C([0,T])$, we solve the stochastic control problem \eqref{eq:mfg} for a representative agent against the fixed environment.
The best response strategy for the representative agent is denoted by $\theta^{*,f},\in \mathbb{U}$ and $V^{*,f}=(V_t^{*,f})_{t\in[0,T]}$ stands for the wealth process in \eqref{eq:Vt-MFG} under the portfolio strategy $\theta^{*,f}=(\theta_t^{*,f})_{t\in[0,T]}$.

\item {\bf Step 2.} We derive an MFE by using the consistency condition, which is equivalent to finding a fixed point $f^*\in\C([0,T])$ taking values on $\R_+$ to the fixed point problem $\Ex[V_t^{*,f}]=\mathrm{v}+\int_0^t f(s)ds$ for all $t\in[0,T]$. Then, the expected MFE is given by $(f^*,\theta^{*,f^*})=(f^*(t),\theta_t^{*,f^*})_{t\in[0,T]}$.
\end{itemize}

\section{Equivalent Problem Formulation for Best Response Control}\label{sec:auxiliary}

To tackle the best response control of the representative agent in problem \eqref{eq:mfg}, we introduce a new controlled state process to replace the original state process. Let $K_t:=\lambda \overline{V}_t+(1-\lambda )Z_t-V_t$ and $L_t:=0\vee \sup_{s\leq t}K_s$ for $t\in[0,T]$. Then, we define a new state process with reflection that $X_t:=L_t-K_t$, and hence 
\begin{align}\label{state-X}
dX_t&=(\theta_t \mu -\lambda f(t)-(1-\lambda)\mu_Z Z_t)dt\\
&\quad+(\theta_t\sigma-(1-\lambda)\sigma_Z Z_t )dW_t+dL_t,~~t\in (0,T]\nonumber
\end{align}
with $X_0=x=((1-\lambda)(\mathrm{v}-z))^+\in\R_+$. In particular, the process $L=(L_t)_{t\in[0,T]}$ is referred to as the local time of $X$, it increases at time $t$ if and only if $X_t=0$, i.e., $L_t=D_t$. We will change the notation from $L_t$ to $L_t^X$ from this point onwards to emphasize its dependence on the new state process $X$ given by \eqref{state-X}.

We then need to solve the following equivalent stochastic control problem, for $(t,x,z)\in[0,T]\times\R_+^2$, 
\begin{align}\label{eq_mfg-2}
&u(t,x,z):=\sup_{\theta\in \mathbb{U}_t} \Ex\left[  -\int_t^{T} e^{-\rho (s-t)}dL_s^X\big|X_t=x,Z_t=z \right],
\end{align}
where the admissible control set $\mathbb{U}_t$ is specified as the set of $\Fx$-adapted real-valued control processes $\theta=(\theta_s)_{s\in[t,T]}$ satisfying $\Ex[\int_t^T |\theta_s|^2ds]<\infty$. Obviously, we have $\mathbb{U}=\mathbb{U}_0$ and the following equivalence between problems \eqref{eq:mfg} and \eqref{eq_mfg-2}:
\begin{align*}
{\rm w}(x,z)=
\begin{cases}
\displaystyle -u(0,(1-\lambda)(\mathrm{v}-z),z),&\text{if}~\mathrm{v}\geq z,\\[0.3em]
\displaystyle -u(0,0,z)+(1-\lambda)(z-\mathrm{v}),&\text{if}~\mathrm{v}< z.
\end{cases}
\end{align*}
It is straightforward to derive the next result:
\begin{lemma}\label{lem:propoertyu}
{\it Let $(t,z)\in[0,T]\times \R_+$ be fixed. Then, the value function $x\mapsto u(t,x,z)$ given by \eqref{eq_mfg-2} is non-decreasing. Moreover, for all $(t,x_1,x_2,z)\in[0,T]\times \R_+^3$, we have
\begin{align}\label{eq:lipvx}
\left|u(t,x_1,z)-u(t,x_2,z)\right|\leq  |x_1-x_2|.
\end{align}}
\end{lemma}

Denote by $\mathcal{D}_T:=[0,T)\times \R_+^2$ and $\overline{\mathcal{D}}_T:=[0,T]\times \R_+^2$.  In what follows, we first assume that the value function $u\in{\C}^{1,2}(\mathcal{D}_T)\cap {\C}\left(\overline{\mathcal{D}}_T\right)$ and $u_{xx} < 0 $ on $\mathcal{D}_T$, which will be verified later. Here, $\C^{1,2}(\mathcal{D}_T)$ denotes the set of real-valued functions on $\mathcal{D}_T$ that are continuously differentiable in $t$ and twice continuously differentiable in $(x,z)$. By the heuristic dynamic programming argument, the HJB equation satisfied by $u$ is given by
\begin{align}\label{HJB}
\begin{cases}
\displaystyle u_t+\sup_{\theta\in\R}\left[u_x\mu\theta+\frac{1}{2}u_{xx}(\sigma\theta-(1-\lambda)\sigma_Zz)^2+u_{xz}(\sigma\theta-(1-\lambda)\sigma_Zz)\sigma_Z z \right] \\[0.6em]
\displaystyle \qquad\qquad-((1-\lambda)\mu_Zz+\lambda f(t)) u_x+\mu_Z zu_z+\frac{1}{2}\sigma_Z^2z^2u_{zz}=\rho u,\quad \forall(t,x,z)\in\mathcal{D}_T,\\[0.6em]
\displaystyle u_x(t,0,z)=1,\quad \forall(t,z)\in[0,T)\times \R_+,\\[0.6em]
\displaystyle u(T,x,z)=0,\quad \forall(x,z)\in\R_+^2.
\end{cases}
\end{align}
The Neumann boundary condition $u_x(t,0,z)=1$ stems from the fact that $L_t^X$ increases if and only if $X_t=0$. 
 By taking the first order condition over the control,  we get, $\forall(t,x,z)\in\mathcal{D}_T$,
\begin{align}\label{HJB-nonlinear}
&u_t-\frac{\mu^2}{2\sigma^2}\frac{u_x^2}{u_{xx}}-\frac{\mu\sigma_Z z}{\sigma}\frac{u_xu_{xz}}{u_{xx}}+\frac{1}{2}\sigma_Z^2z^2\left(u_{zz}-\frac{u_{xz}^2}{u_{xx}}\right)\\
 &\quad+\mu_Z zu_z\left(\frac{(1-\lambda)(\mu_Z\sigma-\mu\sigma_Z)}{\sigma}z+\lambda f(t)\right)u_x =\rho u.\nonumber
\end{align}
The HJB equation \eqref{HJB-nonlinear} is analytically intractable due to its nonlinear terms involving quotients of derivatives. To overcome this, we employ a dual transformation to linearize the problem, converting the original equation into a linear parabolic PDE. The resulting linear structure then enables a probabilistic solution representation via the Feynman-Kac formula, facilitating the study of the regularity using the probabilistic representation.

We first apply the dual transform to linearize the original HJB equation \eqref{HJB-nonlinear}. As a direct result of Lemma~\ref{lem:propoertyu},  $|u_x(t,x,z)|= u_x(t,x,z)\leq 1$ for all $(t,x,z)\in\overline{\mathcal{D}}_T$. Then, we may apply Legendre-Fenchel transform of the solution $u$  with respect to $x$ that, for all $(t,y,z)\in[0,T]\times (0,1]\times\R_+$,
\begin{align}\label{eq:LFT-u}
\hat{u}(t,y,z):=\sup_{x>0}\{u(t,x,z)-xy\}. 
\end{align}
Consequently $u(t,x,z)=\inf_{y\in(0,1]}\{\hat{u}(t,y,z)+xy\}$ for $(t,x,z)\in\overline{\mathcal{D}}_T$. Define $x^*(t,y,z)=u_x(t,\cdot,z)^{-1}(y)$ with $y\mapsto u_x(t,\cdot,z)^{-1}(y)$ being the inverse function of $x\mapsto u_x(t,x,z)$. Thus, $x^*=x^*(t,y,z)$ satisfies the equation $u_x(t,x^*,z)=y$ for $(t,z)\in[0,T]\times \R_+$. 
In view of the dual relation, we get the following dual PDE of \eqref{HJB-nonlinear} that
\begin{align}\label{eq:dual-u}
\begin{cases}
\displaystyle \hat{u}_t+\frac{\mu^2}{2\sigma^2}y^2\hat{u}_{yy}-\frac{\mu\sigma_Z }{\sigma}zy\hat{u}_{yz}+\frac{1}{2}\sigma_Z^2z^2\hat{u}_{zz}+\rho y\hat{u}_y+\mu_Zz\hat{u}_z\\[0.9em]
\displaystyle \quad-\left(\frac{(1-\lambda)(\mu_Z\sigma-\mu\sigma_Z)}{\sigma}z+\lambda f(t)\right)y=\rho \hat{u},\quad \forall(t,y,z)\in[0,T)\times (0,1)\times\R_+,\\[0.8em]
\displaystyle \hat{u}_y(t,1,z)=0,\quad \forall(t,z)\in[0,T)\times \R_+,\\[0.6em]
\displaystyle \hat{u}(T,y,z)=0,\quad \forall(y,z)\in(0,1]\times\R_+.
\end{cases}
\end{align}
For convenience, we further consider the transform  
 $v(t,r,z)=e^{-\rho t}\hat{u}\left(t,e^{-r},z\right)$ for $(t,r,z)\in\overline{\mathcal{D}}_T$. Then,  we arrive at
\begin{align}\label{eq:HJB-v}
\begin{cases}
\displaystyle v_t+\frac{\mu^2}{2\sigma^2}v_{rr}+\frac{\mu\sigma_Z }{\sigma}zv_{rz}+\frac{1}{2}\sigma_Z^2z^2v_{zz}+\left(\frac{\mu^2}{2\sigma^2}-\rho\right) v_r+\mu_Z zu_z\\[0.8em]
\displaystyle\qquad\qquad-\left((1-\lambda)\eta z+\lambda f(t)\right)e^{-\rho t-r}=0,\quad \forall(t,r,z)\in[0,T)\times (0,+\infty)\times\R_+,\\[0.8em]
\displaystyle v_r(t,0,z)=0,\quad \forall(t,z)\in[0,T)\times\R_+,\\[0.6em]
\displaystyle v(T,r,z)=0,\quad \forall(r,z)\in\R_+^2,
\end{cases}
\end{align}
with the parameter $\eta:=\frac{\mu_Z\sigma-\mu\sigma_Z}{\sigma}>0$.

We next study the existence and uniqueness of classical solutions to problem \eqref{eq:HJB-v} using the probabilistic representation. Toward this end, let us define, for $(t, r,z) \in \overline{\mathcal{D}}_T$, 
\begin{align}\label{eq:v}
v(t,r,z)&:=-\lambda\Ex\left[ \int_t^Te^{-\rho s-R^{t,r}_s}f(s) ds\right]-(1-\lambda)\eta\Ex\left[ \int_t^Te^{-\rho s-R^{t,r}_s}Z_s^{t,z} ds\right],
\end{align}
where, for $(t,r,z)\in  [0,T]\times\R_+^2$, the process $(Z_s^{t,z})_{s\in [t,T]}$  is given by \eqref{eq:Zt-MFG} with $Z_t^{t,z}=z$, and the process $(R_s^{t,r})_{s\in [t,T]}$ is  a reflected Brownian motion defined by, for $s\in[t,T]$,
\begin{align}\label{eq:R}
R_s^{t,r}&:=r+\int_t^s \left(\frac{\mu^2}{2\sigma^2}-\rho\right)du+\frac{\mu}{\sigma}\int_t^s dW_u+\int_t^s dL_u^{t,r}\geq 0.
\end{align}
Here, $(L_s^{t, r})_{s\in [t,T]}$ is a continuous and non-decreasing process that increases only on $\{s \in[t, T] ;~R_s^{t, r}=0\}$ with $L_t^{t, r}=0$. 

The well-posedness of problem \eqref{eq:dual-u} is then provided in the next result.
\begin{proposition}\label{thm:dual-u}
{\it Let $f\in {\C}([0,T])$ satisfy $f(t)>0$ for all $t\in[0,T]$. Then, the function  $\hat{u}(t,y,z)$ defined  by $ \hat{u}\left(t,y,z\right)=e^{\rho t}v(t,-\ln y,z)$ for $(t,y,z)\in[0,T]\times (0,1]\times \R_+$,
is the unique classical solution of the dual equation \eqref{eq:dual-u} satisfying $|\hat{u}(t,y,z)|\leq C(1+z)$ for some $C>0$. Moreover, $\hat{u}_y \in\C^{1,2,2}([0,T]\times (0,1]\times\R_+)$, and  for any $(t,z) \in[0, T)\times\R_+ $, the solution $(0,1] \ni y \mapsto \hat{u}(t,y,z)$ is strictly convex.}
\end{proposition}

We next present the verification theorem on the best response control of the representative agent. 
\begin{theorem}[Verification Theorem]\label{thm:verification}
{\it Let $f\in {\C}([0,T])$ satisfy $f(t)> 0$ for all $t\in[0,T]$. Let us introduce the set ${\cal O}_T$ defined by
\begin{align}\label{eq:O}
{\cal O}_T:=\left\{(t,x,z)\in\overline{{\cal D}}_T;~x<x_0(t,z)\right\},
\end{align}
where $x_0(t,z)$ for $(t,z)\in[0,T]\times \R_+$ is defined by
 \begin{align}\label{eq:x0}
 x_0(t,z)&:=\lambda \int_t^{T} f(s)ds+(1-\lambda)\left(e^{\eta (T-t)}-1\right)z.
 \end{align}
Then, we have 
\begin{itemize}
\item[{\rm(i)}] Consider the function $u(t,x,z)$ given  by
\begin{align}\label{eq:sol-u}
&u(t,x,z):=\begin{cases}
\displaystyle \inf_{y\in(0,1]}\{\hat{u}(t,y,z)+yx\},\quad \text{if}~ (t,x,z)\in\mathcal{O}_T,\\[0.8em]
\displaystyle \qquad\qquad0,\quad \text{if}~ (t,x,z)\in \mathcal{O}_T^{c} \cap \overline{\mathcal{D}}_T.
\end{cases}
\end{align}
Consequently, $u\in \C^{1,2}(\mathcal{D}_T)\cap \C(\overline{\mathcal{D}}_T)$ and it is a classical solution of the primal HJB equation \eqref{HJB}.

\item[{\rm(ii)}] Introduce the feedback control function as follows:
\begin{align}\label{eq:optimal-theta}
&\theta^*(t,x,z):=
\begin{cases}
\displaystyle -\frac{\mu}{\sigma^2}\frac{u_x(t,x,z)}{u_{xx}(t,x,z)}-\frac{\sigma_Z}{\sigma}\frac{zu_{xz}(t,x,z)}{u_{xx}(t,x,z)}\\[1em]
\displaystyle\hfill+\frac{(1-\lambda)\sigma_Z}{\sigma}z,\quad\text{if}~ (t,x,z)\in\mathcal{O}_T,\\[1em]
\displaystyle \frac{\mu}{\sigma^2}\lim_{y\downarrow 0} y\hat{u}_{yy}(t,y,z)-\frac{\sigma_Z}{\sigma} \lim_{y\downarrow 0}z \hat{u}_{yz}(t,y,z)\\[1em]
\displaystyle \hfill+\frac{(1-\lambda)\sigma_Z}{\sigma}z,\quad \text{if}~ (t,x,z)\in \mathcal{O}_T^c \cap \overline{\mathcal{D}}_T.
\end{cases}
\end{align}
For $(t,x)\in[0,T]\times \R_+$, consider the controlled state process $X^*=(X^*_s)_{s\in[t,T]}$ that obeys the following reflected SDE:
\begin{align}\label{eq:optimal-SDE}
 X^*_s&=x+\int_t^s\mu \theta^*(\ell,X_{\ell}^*,Z_{\ell}) d\ell+\int_t^s \sigma \theta^*(\ell,X_{\ell}^*,Z_{\ell}) dW_\ell-\int_t^s \lambda f(\ell)d\ell+L_s^{X^*}.
\end{align}
Define $\theta^*_s:=\theta^*(s,X^*_s,Z_s)$ for $s\in[t,T]$. Then, the strategy  $\theta^*=(\theta^*_s)_{s\in [t,T]}\in\mathbb{U}^r_t$ is a best response control of the representative agent. 
\end{itemize}}
\end{theorem}

We refer to $\mathcal{O}_T$ as the underperforming region and $\mathcal{O}_T^{c} \cap \overline{\mathcal{D}}_T$ as the outperforming region. In the underperforming region, even under the optimal portfolio strategy, the agent's wealth falls below the benchmark process. In contrast, in the outperforming region,  the agent can adopt a portfolio strategy such that the resulting wealth process can always outperform the benchmark. Moreover, Theorem~\ref{thm:verification} yields that the next corollary on the semi-explicit form of the optimal portfolio (feedback) function given by
\begin{corollary}\label{coro:optimal-portfolio}
{\it The optimal feedback function $\theta^*(t,x,z)$ given by \eqref{eq:optimal-theta} is always non-negative and  admits the following semi-explicit expression: on $\mathcal{O}_T$,
\begin{align}\label{optimal-portfolio-O}
\theta^*(t,x,z)&=\frac{\lambda \mu}{\sigma^2} \int_{t}^{T}\int_{-\infty}^{r} e^{-\rho (s-t)+\ell}\phi(s-t,\ell,r)f(s)d\ell ds\nonumber\\
&\quad+\frac{(1-\lambda)\eta \mu}{\sigma^2} z\int_{t}^{T}\int_{-\infty}^{r} \exp\left(-(\rho-\kappa)(s-t)+\left(1-\frac{\sigma\sigma_Z}{\mu}\right)\ell\right) \phi(s-t,\ell,r)d\ell ds\nonumber\\
&\quad+\frac{(1-\lambda)\eta\sigma_Z}{\sigma}  \Ex\left[ \int_t^{\tau_{r}^t \wedge T}e^{-\rho (s-t)-R^{t,r}_s+r}Z_s^{t,z} ds\right]+\frac{(1-\lambda)\sigma_Z}{\sigma}z,
\end{align} 
while on $\mathcal{O}_T^c \cap \overline{\mathcal{D}}_T$,
\begin{align}\label{optimal-portfolio-Oc}
\theta^*(t,x,z)&=\frac{(1-\lambda)\sigma_Z}{\sigma}e^{\eta (T-t)}z,
\end{align}
where $r=r(t,x,z):=-\ln y(t,x,z)$ and $x\to y(t,x,z)$ is the inverse function of $y\to-\hat{u}_y(t,y,z)$ for $(t,x,z)\in {\cal O}_T$. The function $\phi(s,x,y)$ is defined by 
\begin{align}\label{eq:phi}
\phi(s,x,y)&:= \frac{2(2 y-x)}{\tilde{\sigma}^2\sqrt{2\tilde{\sigma}^2 \pi s^3}}\exp \left(\frac{\tilde{\mu}}{\tilde{\sigma}} x-\frac{1}{2} \tilde{\mu}^2 s-\frac{(2 y-x)^2}{2\tilde{\sigma}^2 s}\right)
\end{align}
with parameters $\tilde{\mu}:=\frac{\mu}{2\sigma}-\frac{\sigma}{\mu}\rho$ and $\tilde{\sigma}:=-\frac{\mu}{\sigma}$.}
\end{corollary}

\begin{remark}\label{rem:XV}
 In fact, the state processes of the primal control problem \eqref{eq:mfg} and the auxiliary control problem \eqref{eq_mfg-2} satisfy the following relationship:
\begin{align}\label{eq:state}
X_t &= V_t - \lambda \overline{V}_t - (1-\lambda)Z_t + \sup_{s\leq t} \left(-V_s + \lambda \overline{V}_s + (1-\lambda)Z_s\right)^+,~~\forall t \in [0,T]. 
\end{align}
Hence, the auxiliary state process $X=(X_t)_{t \in [0,T]}$ can be constructed from the state-benchmark processes $(V,Z)=(V_t, Z_t)_{t \in [0,T]}$. This reveals that the optimal control $\theta^*=(\theta_t^*)_{t \in[0,T]}$ in Theorem~\ref{thm:verification} and Corollary~\ref{coro:optimal-portfolio} exhibits a path-dependent structure in terms of the wealth process $V$ and the benchmark process $Z$. Such path-dependent feature makes the direct PDE analysis intractable. This observation underscores the main advantage of working with the auxiliary state process $X$: it significantly simplifies the problem and enables us to derive the optimal control $\theta^*$ in feedback form in terms of $X$.
\end{remark}

\section{Mean Field Equilibrium}\label{sec:MFG}

This section verifies the consistency condition provided in Definition \ref{def:MFG} for the deterministic function $t\mapsto f(t)$ that
\begin{align}\label{eq:consistency1}
\mathrm{v}+\int_0^t f(s)ds=\mathbb{E}\left[V_t^{\theta^{*,f}}\right], \quad \forall t\in [0,T],
\end{align}
and establishes the existence of a positive function $f^*\in\C([0,T])$ satisfying \eqref{eq:consistency1}. Consequently, the MFE is given by $(f^*,\theta^{*,f^*})=(f^*(t),\theta_t^{*,f^*})_{t\in[0,T]}$ characterized in Definition~\ref{def:MFG}. On the other hand, it follows from \eqref{eq:Vt-MFG} that $V_t^{\theta^{*,f}}=\mathrm{v}+\int_0^t \theta_s^{*,f} \mu ds+\int_0^t \theta_s^{*,f}\sigma dW_s$ for $t\in[0,T]$. Then, taking the expectation on both sides of this equation, we have
\begin{align}\label{eq:consistency2}
\Ex\left[V_t^{\theta^{*,f}}\right]=\mathrm{v}+\int_0^t \Ex\left[\mu\theta_s^{*,f}\right]ds,\quad \forall t\in[0,T].
\end{align}
Thus, in view of \eqref{eq:consistency1} and \eqref{eq:consistency2}, to verify the consistency condition, it suffices to find a positive function $f\in\C([0,T])$ such that
\begin{align}\label{eq:consistency-f}
f(t)= \mu\Ex\left[\theta_t^{*,f}\right],\quad\forall t\in[0,T],
\end{align}
where $\theta^{*,f}=(\theta_t^{*,f})_{t\in[0,T]}$ is the optimal (feedback) portfolio strategy provided in Theorem \ref{thm:verification} with $f\in\C([0,T])$. Toward this end, we proceed with the following three steps:
\begin{itemize} 
\item {\bf Step 1.}\quad Given a positive function $f\in\C([0,T])$, we introduce the so-called reflected dual process and characterize the optimal (feedback) strategy $\theta^{*,f}$ in view of the reflected dual process.
\item {\bf Step 2.}\quad We prove that, start from the dual variable, there exists some positive function $f\in\C([0,T])$ satisfying the consistency condition \eqref{eq:consistency-f}.

\item {\bf Step 3.}\quad We establish the continuous map from the dual variable to the primal state variable, which helps us to find a positive fixed point $f^*\in\C([0,T])$ satisfying the consistency condition \eqref{eq:consistency-f} for primal state variables $(x,z)\in\R_+^2$ . 
\end{itemize}

Next, given a positive function $f\in\C([0,T])$, we will show that the partial derivative $u_x\in \C^{1,2}({\cal O}_T)$ with $u\in \C^{1,2}(\mathcal{D}_T)\cap \C(\overline{\mathcal{D}}_T)$ being given by \eqref{eq:sol-u} and then introduce the reflected dual process in Lemma~\ref{lem:dual-process} below. We first have
\begin{lemma}\label{lem:u_x}
{\it The partial derivative of value function $u$ given by \eqref{eq:sol-u} with respect to $x$ satisfies that $u_x\in \C^{1,2}({\cal O}_T)$.}
\end{lemma}
Consequently, we have
\begin{lemma}\label{lem:dual-process}
{\it Recall the process $X^*=(X_t^*)_{t\in[0,T]}$ given by \eqref{eq:optimal-SDE}. Introduce the process $Y=(Y_t)_{t\in [0,T]}$ by $Y_t:=u_x(t,X_t^*,Z_t)$ for all $t\in[0,T)$ and $Y_T:=\lim_{t\uparrow T}Y_t$ with $(X_0^*,Z_0)=(x,z)\in\R_+^2$. Define the stopping time $\tau_0:=\inf\{t\in[0,T];~X_t^*\geq x_0(t,Z_t)\}\wedge T$ ($\inf\varnothing=+\infty$). Then 
\begin{itemize}
\item[{\rm(i)}] if $(0,x,z)\in {\cal O}_T$, then $\tau_0=T$, $\Px$-a.s.. If $(0,x,z)\in {\cal O}_T^c\cap \overline{\mathcal{D}}_T$, then $\tau_0=0$, $\Px$-a.s., and for any $t\in[0,T]$, $(t,X_t^*,Z_t)\in  {\cal O}_T^c\cap \overline{\mathcal{D}}_T$, $\Px$-a.s.;
\item[{\rm (ii)}] for $(0,x,z)\in {\cal O}_T$, the process $Y$ taking values on $(0,1]$ satisfies the following reflected SDE:
\begin{align}\label{eq:Yt}
dY_t= \rho Y_tdt-\frac{\mu}{\sigma}Y_tdW_t-dL_t^Y,~~t\in(0,T],
\end{align}
where the process $L=(L_t^Y)_{t\in[0,T]}$ is a continuous and non-decreasing process (with $L_t^{Y}=0$) that increases on the set $\{t\geq 0;~Y_t=1\}$ only.
\end{itemize}}
\end{lemma}

\begin{remark}\label{rem:injection}
Lemma \ref{lem:dual-process} shows that the underperforming region ${\cal O}$ and outperforming region ${\cal O}_T^c\cap \overline{\mathcal{D}}_T$ are not interconnected. In other words, if the initial state $(0,x,z)\in {\cal O}_T$, $(t,X_t^*,Z_t)\in{\cal O}_T$, $\Px$-a.s., for any $t\in[0,T]$; if $(0,x,z)\in {\cal O}_T^c\cap \overline{\mathcal{D}}_T$, then for any $t\in[0,T]$, $(t,X_t^*,Z_t)\in  {\cal O}_T^c\cap \overline{\mathcal{D}}_T$, $\Px$-a.s.. Whether the wealth process falls below the benchmark completely depends on the initial wealth level of the fund manager and  initial level of the benchmark process.
\end{remark}

In view of Theorem \ref{thm:verification}, Corollary \ref{coro:optimal-portfolio} and the reflected dual process $Y=(Y_t)_{t\in[0,T]}$, the optimal portfolio strategy $\theta^{*,f}=(\theta_t^{*,f})_{t\in[0,T]}$  admits the following expression.
\begin{corollary}\label{coro:optimal-portfolio-dual}
{\it The optimal portfolio strategy $\theta_t^{*,f}=(\theta^{*,f})_{t\in[0,T]}$ defined in Theorem \ref{thm:verification} admits the following expression,  on $\mathcal{O}_T$,
\begin{align}\label{optimal-portfolio-daul-O}
&\theta^{*,f}_t=\frac{\lambda \mu}{\sigma^2} \int_{t}^{T}\int_{-\infty}^{R_t^{0,r}} e^{-\rho (s-t)+x}\phi(s-t,x,R_t^{0,r})f(s) ds\nonumber\\
&\quad+\frac{(1-\lambda)\eta\sigma_Z}{\sigma}  \varphi(t,R_t^{0,r},Z_t^{0,z})+\frac{(1-\lambda)\sigma_Z}{\sigma}Z_t^{0,z}\nonumber\\
&\quad+\frac{(1-\lambda)\eta \mu}{\sigma^2} Z_t^{0,z}\int_{t}^{T}\int_{-\infty}^{R_t^{0,r}} \exp\left(-(\rho-\kappa)(s-t)+\left(1-\frac{\sigma\sigma_Z}{\mu}\right)x\right)\phi(s-t,x,R_t^{0,r})dx ds,
\end{align} 
on $\mathcal{O}_T^c \cap \overline{\mathcal{D}}_T$,
\begin{align}\label{optimal-portfolio-dual-Oc}
\theta^{*,f}_t&=\frac{(1-\lambda)\sigma_Z}{\sigma}e^{\eta (T-t)}Z_t^{0,z},
\end{align}
where the function $\varphi:[0,T]\times \R_+^2\mapsto \R$ is defined by
\begin{align}\label{eq:varphi}
\varphi(t,r,z):=\Ex\left[ \int_t^{\tau_{\hat{r}}^t \wedge T}e^{-\rho (s-t)-R^{t,r}_s+r}Z_s^{t,z} ds\right].
\end{align}
Here, recall the reflected process $(R_t^{0,r})_{t\in [0,T]}$ given by \eqref{eq:R} with $r=-\ln(u_x(0,x,z))$ and the function $\phi$ given by \eqref{eq:phi}.}
\end{corollary}

In what follows, we use the notation ${\cal O}_T^f$ to highlight the dependence of the underperforming region ${\cal O}_T$ defined by \eqref{eq:O} on the function $f\in\C([0,T])$. Based on Corollary \ref{coro:optimal-portfolio-dual}, for any $(x,z)\in \R_+^2$ , let us consider the following mapping ${\cal I}:\C([0,T])\mapsto \C([0,T])$ given by
\begin{align}\label{eq:Phi}
{\cal I} f(t):=
\begin{cases}
\displaystyle \int_t^T G(r,s,t)f(s)ds+H(r,z,t), & \text{if}~(0,x,z)\in {\cal O}_T^f,\\[1.2em]
\displaystyle \frac{(1-\lambda)\sigma_Z\mu}{\sigma}e^{\eta (T-t)+\mu_Z t}z, &\text{if}~(0,x,z)\in (\mathcal{O}_T^f)^c \cap \overline{\mathcal{D}_T},
\end{cases}
\end{align}
where $r=r^f(0,x,z):=-\ln(u_x^f(0,x,z))$, the functions $G(r,s,t)$ and $H(r,z,t)$ are given by, for $(r,z)\in\R_+^2$ and $0\leq t\leq s\leq T$,
\begin{align}
&G(r,s,t):=\frac{\lambda \mu^2}{\sigma^2} \Ex\left[\int_{-\infty}^{R_t^{0,r}} e^{-\rho (s-t)+x}\phi(s-t,x,R_t^{0,r})dx \right],\label{eq:G}\\
&H(r,z,t)\label{eq:H}\nonumber\\
&:=\frac{(1-\lambda)\eta \mu^2}{\sigma^2} \Ex\left[Z_t\int_{t}^{T}\int_{-\infty}^{R_t^{0,r}} \exp\left(-(\rho-\kappa)(s-t)+\left(1-\frac{\sigma\sigma_Z}{\mu}\right)x\right) \phi(s-t,x,R_t^{0,r})dx ds\right]\nonumber\\
&\quad+\frac{(1-\lambda)\eta\sigma_Z\mu}{\sigma}  \Ex\left[\varphi(t,R_t^{0,r},Z_t^{0,z})\right]+\frac{(1-\lambda)\sigma_Z\mu}{\sigma}e^{\mu_Z t}z.
\end{align}
As a result, to verify the consistency condition \eqref{eq:consistency-f}, it is sufficient to show that the  mapping ${\cal I}$ has a fixed point $f^*\in\C([0,T])$ satisfying $f^*(t)>0$ for all $t\in[0,T]$. 
\begin{remark}
 A main challenge of this problem lies in the piecewise structure of the mapping \eqref{eq:Phi}, whose expression and the boundary both dependent on the input $f$. The free boundary that separates the underperforming region from the outperforming region is not fixed but evolves with the input $f$, creating a moving boundary problem in the fixed point analysis.
\end{remark}

We then need the next two results.

\begin{lemma}\label{lem:fixpoint-r}
{\it Let $(r,z)\in \R_+^2$ be fixed. Define the mapping ${\cal J}:\C([0,T])\mapsto \C([0,T])$ by, for any $f\in\C([0,T])$ and $t\in[0,T]$,
\begin{align}\label{eq:map-J}
{\cal J}f(t):=\int_t^T G(r,s,t) f(s) ds+H(r,z,t).
\end{align}
Then, ${\cal J}:\C([0,T])\mapsto \C([0,T])$ has a unique fixed point $f\in \C([0,T])$ satisfying $f(t)>0$ for all $t\in[0,T]$.}
\end{lemma}

\begin{lemma}\label{lem:continuity-f}
{\it Let $z\in \R_+$ be fixed. For $(t,r)\in[0,T]\times\R_+$, denote by $f(t;r)$ the unique fixed point given in Lemma \ref{lem:fixpoint-r}. Then, the function $(t,r)\mapsto f(t;r)$ is continuous.}
\end{lemma}

Lemma \ref{lem:fixpoint-r} establishes the existence and uniqueness of the fixed point $f$ to the mapping \eqref{eq:map-J}; while Lemma \ref{lem:continuity-f} shows the continuity of the fixed point $f$ given in Lemma \ref{lem:fixpoint-r} with respect to $(t,r)$. Using these auxiliary results, we have the existence of the MFE as a main result of this paper.
\begin{theorem}[MFE]\label{thm:MF-equilibrium}
Let $(x,z)\in\R_+^2$. Then, there exists a fixed point $f^*\in\C([0,T])$ to the mapping \eqref{eq:Phi} satisfying $f^*(t)>0$ for all $t\in[0,T]$. Let $\theta^{*,f^*}=(\theta_t^{*,f^*})_{t\in[0,T]}$ be defined in Theorem \ref{thm:verification} for the fixed point $f^*\in\C([0,T])$. Then, $(f^*,\theta^{*,f^*})=(f^*(t),\theta_t^{*,f^*})_{t\in[0,T]}$ is a MFE.
\end{theorem}

\begin{proof}
For $z\in\R_+$, define $\hat{x}_0(z)$ by
\begin{align}\label{eq:hat-x}
\hat{x}_0(z):=(1-\lambda)\left(\lambda\left(e^{\mu_Z T}-e^{\eta T}\right)+e^{\eta T}-1\right)z>0.
\end{align} 
Then, for any $(x,z)\in\{(x,z)\in \R_+^2;~x\geq \hat{x}_0(z)\}$,  let us define that
\begin{align*}
f(t)&:=\frac{(1-\lambda)\sigma_Z\mu}{\sigma}e^{\eta (T-t)}\Ex\left[Z_t^{0,z}\right]=\frac{(1-\lambda)\sigma_Z\mu}{\sigma}e^{\eta (T-t)+\mu_Z t}z,\quad \forall t\in[0,T].
\end{align*}
Then, $f\in\C([0,T])$ with $f(t)>0$ for all $t\in[0,T]$. Moreover, It is not difficult to verify that $f$ defined above satisfies the consistency condition \eqref{eq:consistency-f}.

Next, we consider the case where $(x,z)\in\{(x,z)\in \R_+^2;~x< \hat{x}_0(z)\}$. It follows from Lemma \ref{lem:fixpoint-r} that, for $(r,z)\in \R_+^2$, the mapping \eqref{eq:map-J} has a unique fixed point $f(\cdot;r)\in\C([0,T])$ satisfying $f(t;r)>0$ for all $t\in[0,T]$. In view of the dual relationship, we have that
\begin{align}\label{eq:x-r}
x(r)&=-\hat{u}_y(0,y,z)=e^{r}v_r(0,r,z)\nonumber\\
&=\lambda e^{r}\mathbb{E}\left[\int_0^{\tau_r^0 \wedge T} e^{-\rho s-R_s^{0, r}} f(s;r)ds\right]+(1-\lambda)\eta  e^{r}\Ex\left[ \int_0^{\tau_r^0 \wedge T}e^{-\rho s-R^{0,r}_s}Z_s^{0,z} ds\right].
\end{align}
This, together with Lemma \ref{lem:continuity-f}, implies that the mapping $r \mapsto x(r)$ is continuous and that
\begin{align}\label{eq:lim-xr}
\lim_{r\downarrow 0}x(r)=0,\quad \lim_{r\uparrow +\infty}x(r)=\hat{x}_0(z).
\end{align}
Consequently, by the continuity of $r\to x(r)$ and \eqref{eq:lim-xr}, for any $\tilde{x}\in [0,\hat{x}_0(z))$, there exists some $r_{\tilde{x}}\in\R_+$ such that $x(r_{\tilde{x}})=\tilde{x}$ and also a unique fixed point $f(\cdot;r_{\tilde{x}})\in\C([0,T])$ of the mapping \eqref{eq:map-J} satisfying $f(\cdot;r_{\tilde{x}})>0$.
Then, for $(x,z)\in\R_+^2$, let us define that, for $t\in[0,T]$,
\begin{align}\label{eq:equilibrium-f}
f^*(t)=
\begin{cases}
\displaystyle f(t;r_x),&\text{if}~x\in[0,\hat{x}_0(z)),\\[1em]
\displaystyle \frac{(1-\lambda)\sigma_Z\mu}{\sigma}e^{\eta (T-t)+\mu_Z t}z,&\text{if}~x \in[\hat{x}_0(z),+\infty).
\end{cases}
\end{align}
Obviously, $f^*\in\C([0,T])$ and $f^*(\cdot)>0$. As a consequence, we can easily verify that the function $f^*$ given by \eqref{eq:equilibrium-f} is a fixed point to the mapping \eqref{eq:Phi}. This yields that $f^*$ given by \eqref{eq:equilibrium-f} satisfies the consistency condition \eqref{eq:consistency1}, and hence $(f^*,\theta^{*,f^*})$ is a MFE. Thus, the proof of the theorem is completed.
\end{proof}

\section{Discussions on MFE}\label{sec:numerical}
We present in this section some quantitative properties and numerical results of the MFE obtained in Theorem \ref{thm:MF-equilibrium}. We first recall that, for $(z,T)\in\R_+\times(0,+\infty)$, the outperforming threshold of initial wealth is given by
\begin{align}\label{eq:threshold}
\hat{x}_0(z)=(1-\lambda)\left\{\lambda\left(e^{\mu_Z T}-e^{\eta T}\right)+e^{\eta T}-1\right\}z.
\end{align} 
If the representative agent's initial wealth level $x\geq \hat{x}_0(z)$, then the portfolio feedback function $\theta^{*,f^*}(t,x,z)$ admits the following explicit expression, for $(t,x,z)\in[0,T]\times [\hat{x}_0(z),+\infty)\times\R_+$,
\begin{align}\label{MFG-portfolio-Oc}
\theta^{*,f^*}(t,x,z)&=\frac{(1-\lambda)\sigma_Z\mu}{\sigma}e^{\eta (T-t)}z.
\end{align}
In this case, the representative agent's wealth process under optimal portfolio strategy consistently outperforms the benchmark.

We next discuss two extreme cases.  As the competition weight parameter $\lambda$ tends towards $1$, the benchmark process for agent $i$ becomes  $\overline{V}_t^n$. This indicates that the agent is primarily concerned with their relative performance compared to other peers, resulting in the threshold  $\hat{x}_0(z)\equiv 0$ and optimal portfolio $\theta^{*,f^*}(t,x,z)\equiv 0$. Without the pressure from the market index performance, the agent refrains from making any investments. Conversely, as the competition weight parameter $\lambda$ tends towards $0$, the benchmark process for agent $i$  becomes  $Z_t^i$.  In this scenario, the agent's goal is to outperform the market index benchmark $Z_t^i$, disregarding the wealth performance of other agents. Consequently, the MFG problem reduces to the stochastic control problem discussed in Section 6 of  \cite{BoLiaoYu21} with $\hat{x}_0(z)=(e^{\eta T}-1)z$ and $\theta^{*,f^*}(t,x,z)=\sigma_Z\mu/\sigma e^{\eta (T-t)}z$.

We next consider the sensitivity analysis on the return rate parameter $\mu_Z$ of the market index process. From \eqref{eq:threshold}, \eqref{MFG-portfolio-Oc}, and the relationship $\eta=\mu_Z-\mu\sigma_Z/\sigma$, it follows that $\partial \hat{x}_0/\partial \mu_Z>0$,  and  $\partial \theta^{*,f^*}/\partial \mu_Z>0$, which implies that both the threshold  and the portfolio feedback function increase with respect to $\mu_Z$. When the market index benchmark process has a higher return, the agent will require a higher level of initial wealth and will invest more in the risky asset to outperform the targeted benchmark.

 If the initial wealth level $x< \hat{x}_0(z)$, even under the optimal portfolio strategy, the agent's wealth falls below the benchmark process. In this case, we can numerically calculate the MFE and value function. Our numerical method proceeds through the following key steps by leveraging the duality  established in our theoretical analysis:
\begin{itemize}
    \item[(i)] Parameter initialization: we initialize all model parameters $(\mu,\sigma,\mu_Z,\sigma_Z,\lambda,\rho)$, the time horizon $T$, and the initial values of the wealth and benchmark processes $(x_0,z_0)$.
   \item [(ii)] Initial outperforming assessment: we compute the outperforming threshold $\hat{x}_0(z)$ via  \eqref{eq:threshold}. If $x_0 \geq \hat{x}_0(z)$, the agent starts in the outperforming region. In this case, the optimal portfolio feedback function $\theta^{*,f^*}$ is computed directly via  \eqref{MFG-portfolio-Oc}, and the value function satisfies $v(t,x,z)=0$. Otherwise, we proceed to the following steps.
  \item[(iii)] Fixed point iteration for dual variable: for a candidate dual variable $r > 0$, we compute the corresponding fixed-point function $f(r;\cdot)$ from  \eqref{eq:map-J} in  Lemma \ref{lem:fixpoint-r} through an iterative procedure. This step requires evaluating the functions $G(r,s,t)$ and $H(r,z,t)$, which involve computing expectations of the reflected process $R=(R_t)_{t\in[0,T]}$. We employ the Skorokhod representation:
   {\small
    \begin{align*}
    \begin{cases}
    \displaystyle R_s^{t,r}= r + \int_t^s \left(\frac{\mu^2}{2\sigma^2} - \rho\right) du + \frac{\mu}{\sigma}\int_t^s dW_u + \int_t^s dL_u^{t,r} ,\\
    \displaystyle L_s^{t,r} = \sup_{t \leq \ell \leq s} \left(-r - \int_t^\ell \left(\frac{\mu^2}{2\sigma^2} - \rho\right) du - \frac{\mu}{\sigma}\int_t^\ell dW_u\right)^+,
    \end{cases}
    \end{align*}}and implement this using Monte Carlo simulation to generate sample paths of both the reflected process $R$ and the benchmark process $Z$, enabling accurate computation of the required expectations.
 \item[(iv)] Recovery of primal state variable: we repeat step (iii) for different values of $r$ and construct the function $r \mapsto x(r)$ defined in \eqref{eq:x-r}. We then identify the specific value $r_0$ such that $x(r_0) = x_0$, thereby recovering the original state variable from the dual variable through the dual relationship.
 \item[(v)] Computation of MFE and value function: using the identified $r_0$, we compute the fixed-point function $f(\cdot;r_0)$ as in step (iii). This allows us to fully characterize the MFE and compute the value function by invoking Theorem \ref{thm:verification}.
\end{itemize}

Let us set parameters $(\mu,\sigma,\mu_Z,\sigma_Z,\lambda,\rho)=(0.1,0.1,$ $0.2,$ $0.1,0.2,1)$, the time horizon $T=1$ and the initial level of  benchmark process  $z_0=20$. Figure \ref{fig:f1} plots the function $r\mapsto x(r)$, which shows that there exists a unique $r_{\tilde{x}}$ such that $x(r_{\tilde{x}})=\tilde{x}$, i.e., the MFE is unique under the given model parameters. Figure \ref{fig:f2} plots the  fixed point  function $t\mapsto f^*(t)$ under different initial level of wealth process.

\begin{figure}[htbp]
\centering
\includegraphics[width=7cm]{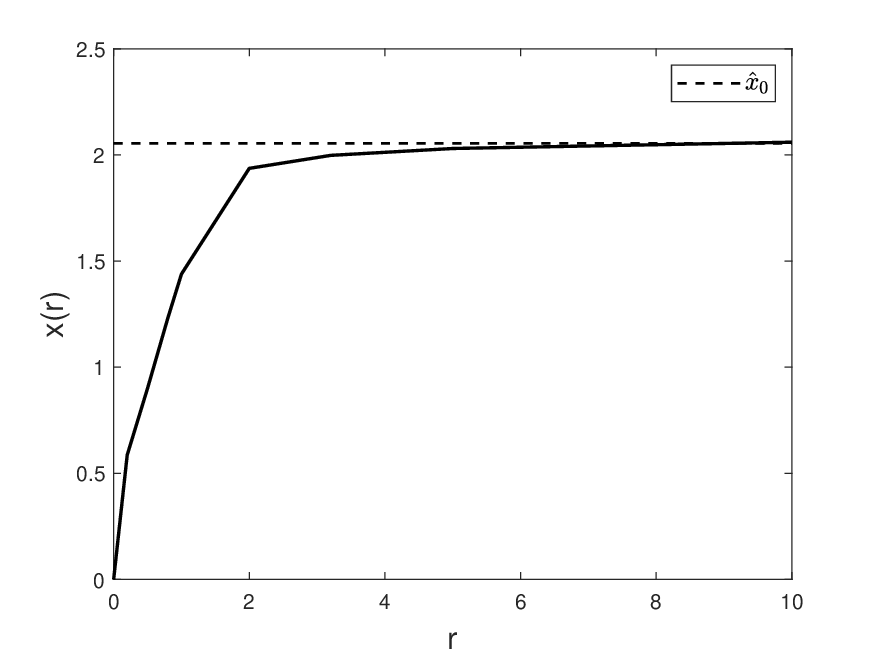}
\caption{  The function $r\mapsto x(r)$.}\label{fig:f1}
\end{figure}

\begin{figure}[htbp]
\centering
\includegraphics[width=7cm]{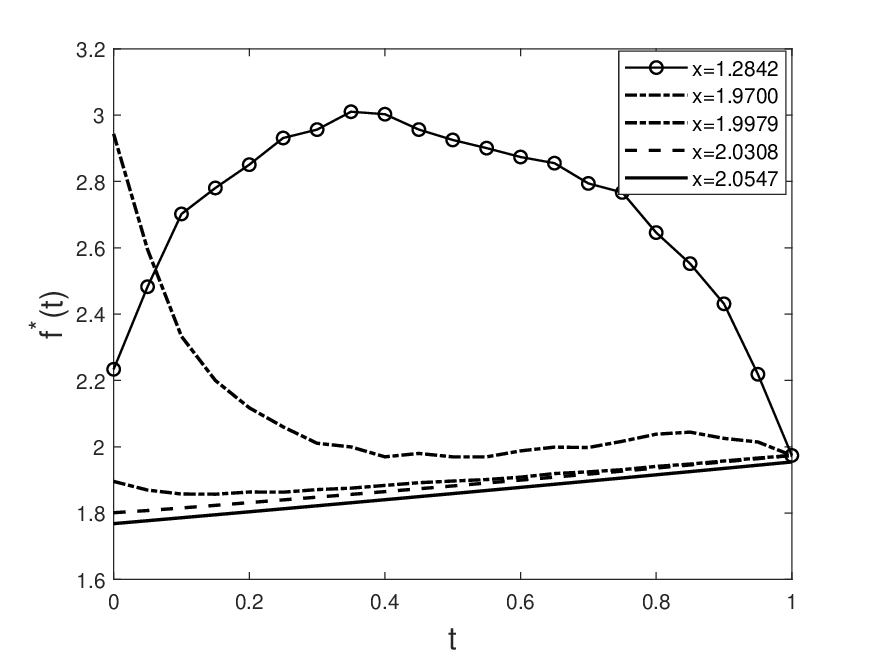}
\caption{ The  fixed point  function $t\mapsto f^*(t)$.}\label{fig:f2}
\end{figure}

\begin{figure}[htbp]
\centering
\includegraphics[width=7cm]{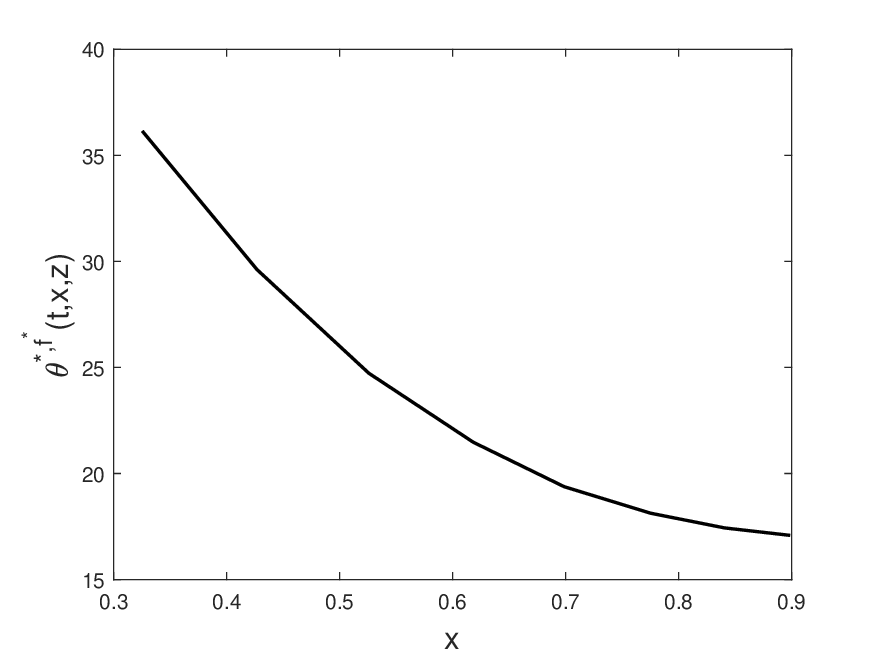}
\caption{The portfolio feedback function $x\mapsto \theta^{*,f^*}(t,x,z)$. The parameters are set to be $(x_0,z_0)=(2.0308,20)$, $(t,z)=(0.5,20)$.}\label{fig:f3}
\end{figure}

\begin{figure}[htbp]
\centering
\includegraphics[width=7cm]{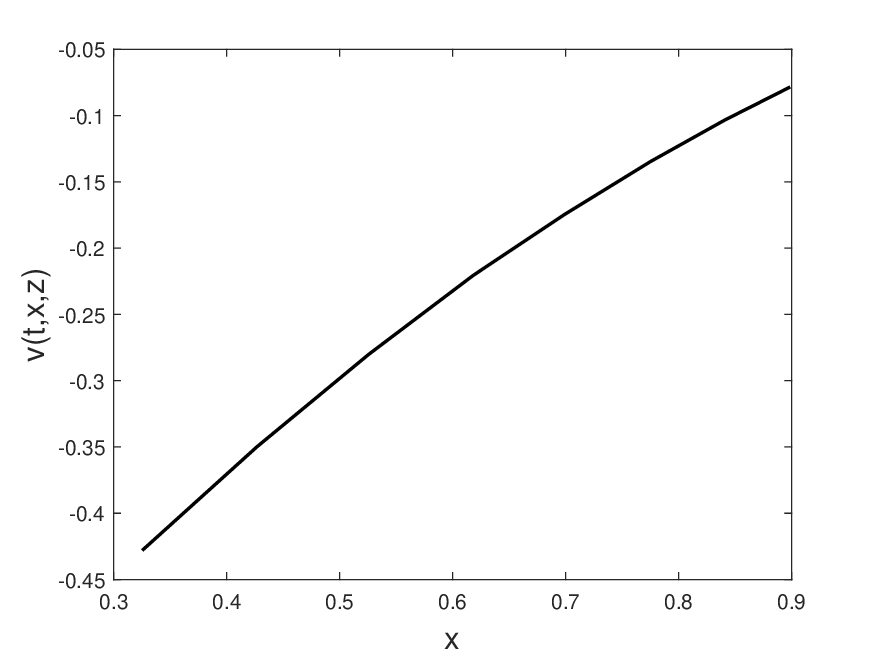}
\caption{ The  value function $x\mapsto v(t,x,z)$. The parameters are set to be $(x_0,z_0)=(2.0308,20)$, $(t,z)=(0.5,20)$.}\label{fig:f4}
\end{figure}
Then, we plot the portfolio feedback function $x\mapsto \theta^{*,f^*}(t,x,z)$ in Figure \ref{fig:f3} and value function $x\mapsto v(t,x,z)$ in Figure \ref{fig:f4}. In particular, it is interesting to see that $\theta^{*,f^*}(t,x,z)$ is a decreasing function of $x$ for the fixed $t$ and $z$, which seems counter-intuitive. However, we note that the large value of $x$ indicates that the current wealth level of all agents or the population is also high, i.e., $f^*(t)$  is very large as well. As a result, in the equilibrium state, it has a larger chance for the agent to falls below the benchmark process. more frequently due to the tracking constraint. To reduce the largest shortfall with reference to the benchmark process, the representative agent in the equilibrium will strategically reduce the allocation in the risky asset when $x$ becomes large to pull down $f^*(t)$  to lower the growth rate of the benchmark process. From Figure \ref{fig:f2}, it can be observed that under different initial values of the state process $X$, although $f^*(t)$ is always non-negative, its monotonicity in terms of $t$ may vary significantly under the impact of initial wealth. In particular, in the situation when $f^*(t)$  climbs to a very large level at some time $t$, we can see from Figure \ref{fig:f2} that $f^*(t)$  will start to decrease afterwards, which echos from the observation in Figure \ref{fig:f3} that the portfolio strategy under MFE is decreasing in $x$ such that the resulting $f^*(t)$ under MFE will decline when it hits a high level, that is, the representative agent will choose the portfolio strategy to maintain the growth rate of the wealth around a reasonable level.

We next conduct some numerical examples on the sensitivity analysis with respect to some model parameters. Recall that the value function is an equivalent characterization of the  expectation of the largest shortfall. We illustrate in Figures \ref{fig:f5} and Figure \ref{fig:f6} the sensitivity of the  expectation of thelargest shortfall as well as the optimal portfolio feedback function with respect to the competition weight parameter $\lambda$. All other model parameters are kept as in the baseline specification before. Consistent with the outperforming case, the portfolio feedback function is decreasing in $\lambda$. As the competition weight increases, the agent places greater emphasis on relative performance against peers and correspondingly less on outperforming the market index in absolute terms. This shift in objective leads to a more conservative optimal allocation to the risky asset: the manager reduces exposure to  risk and relies less on aggressive trading to generate outperformance. In parallel, Figure \ref{fig:f5} shows that a smaller competition weight induces the manager to incur a larger expected largest shortfall relative to the benchmark;; as $\lambda$ rises, the expected largest shortfall  decline. Relaxed competition concerns  against peers (lower $\lambda$) encourage risk-taking  that may boost short-term outperformance but raise the probability of drawdowns, which are reflected in a larger expected largest shortfall.. These comparative statics suggest that tuning $\lambda$ can serve as a policy lever. Institutions seeking lower risk budgets should adopt higher competition weights, whereas those prioritizing  the performance of market index may allow lower $\lambda$, accepting a higher expected largest shortfall.

We also examine how the  expectation of the largest shortfall and the optimal portfolio feedback respond to changes in the volatility parameter $\sigma_Z$ of the market index process. Figures \ref{fig:f7} and \ref{fig:f8} plot these sensitivity results.
The most prominent finding, illustrated in Figure \ref{fig:f7}, is that higher market index volatility $\sigma_Z$ leads to reduced expected largest shortfall. Recall that the effective benchmark is a linear combination $\lambda \overline{V}_t + (1-\lambda)Z_t$. When $\sigma_Z$ increases, the component $Z_t$ becomes more volatile, introducing greater variability into the benchmark itself. Faced with a more volatile benchmark, the agent's optimal strategy naturally adapts to maintain a more conservative position relative to the benchmark. Mathematically, this translates to a decrease in the activity of the reflection process $L_t$, which represents  largest shortfall. The manager's adjusted strategy results in fewer instances where the portfolio value approaches the critical threshold requiring intervention. 

Figure \ref{fig:f8} further reveals the non-monotonic property of the optimal portfolio feedback function with respect to $\sigma_Z$. Specifically, at high wealth levels, an increase in the market index volatility $\sigma_Z$ can lead to more aggressive portfolio weights, as the manager seeks to exploit higher volatility for potential outperformance, thereby tracking the market index and the strong performance of peers. Conversely, at low wealth levels, the same increase in $\sigma_Z$ prompts a reduction in portfolio risk to avoid incurring a costly largest shortfall. In this scenario, the agent reduces investment to mitigate risk and may accept a large shortfall to meet the benchmark requirement. From a practical perspective, these findings suggest that agent operating in more volatile market environments may benefit from a lower expected largest shortfall. However, this comes at the potential cost of reduced opportunity to significantly outperform the benchmark during market upswings. The optimal strategy thus represents a careful balancing act between protection and performance that adapts rationally to the market condition.
\begin{figure}[htbp]
\centering
\includegraphics[width=7cm]{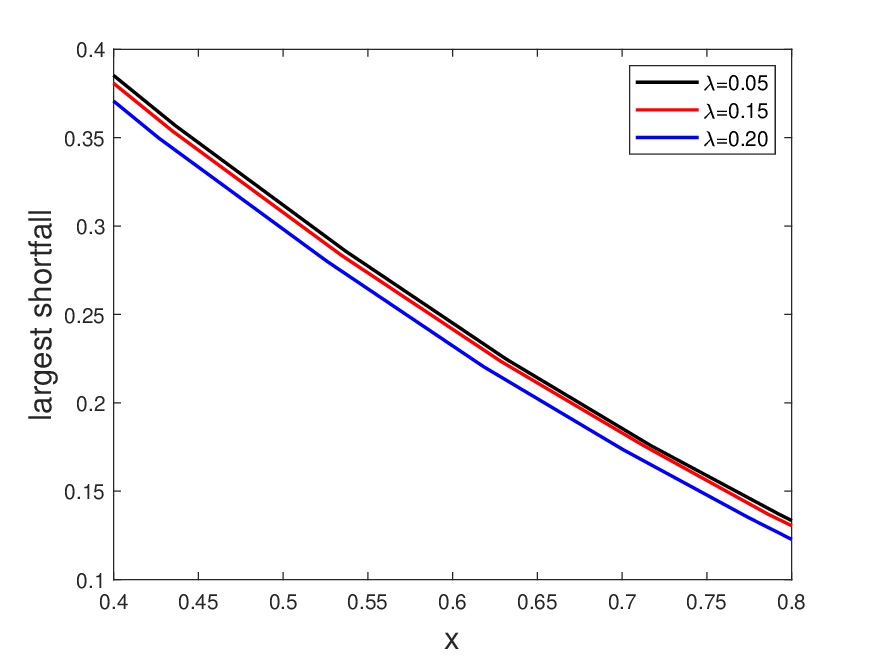}
\caption{ The expectation of the  largest shortfall. The parameters are set to be $(x_0,z_0)=(2.02,20)$, $(t,z)=(0.5,20)$.}\label{fig:f5}
\end{figure}

\begin{figure}[htbp]
\centering
\includegraphics[width=7cm]{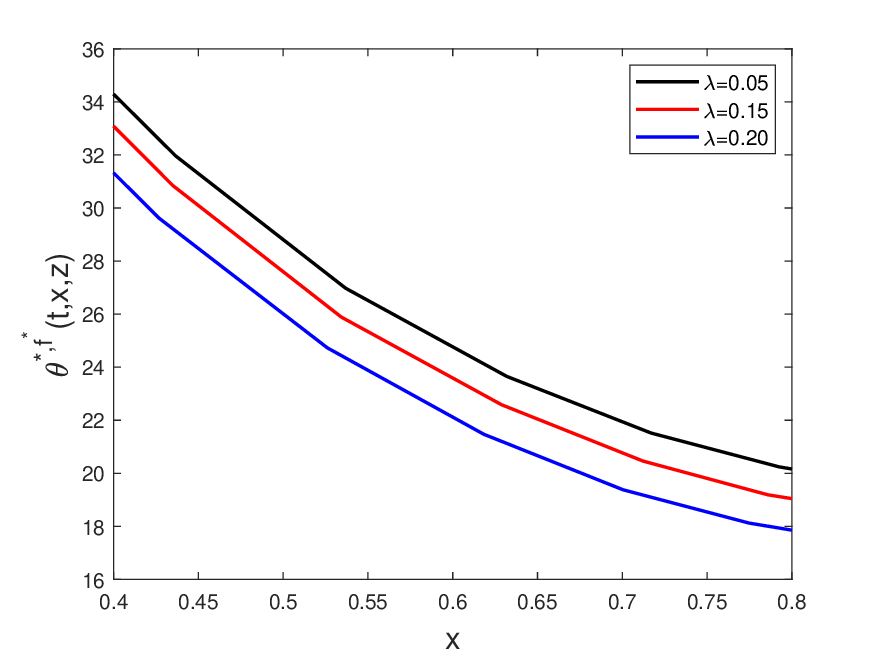}
\caption{ The portfolio feedback function $x\mapsto \theta^{*,f^*}(t,x,z)$. The parameters are set to be $(x_0,z_0)=(2.02,20)$, $(t,z)=(0.5,20)$.}\label{fig:f6}
\end{figure}

\begin{figure}[htbp]
\centering
\includegraphics[width=7cm]{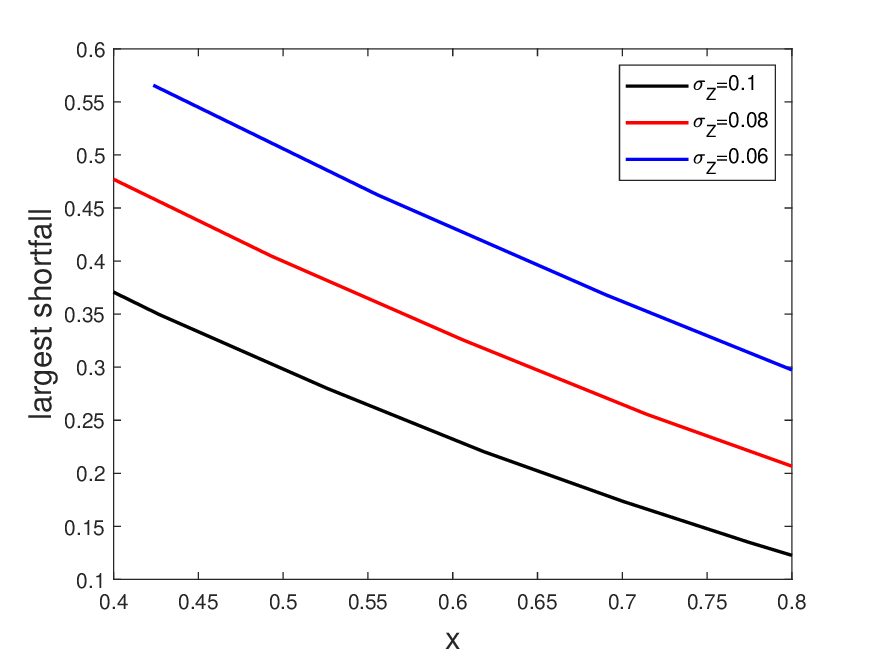}
\caption{ The expectation of the  largest shortfall. The parameters are set to be $(x_0,z_0)=(2.02,20)$, $(t,z)=(0.5,20)$.}\label{fig:f7}
\end{figure}

\begin{figure}[htbp]
\centering
\includegraphics[width=7cm]{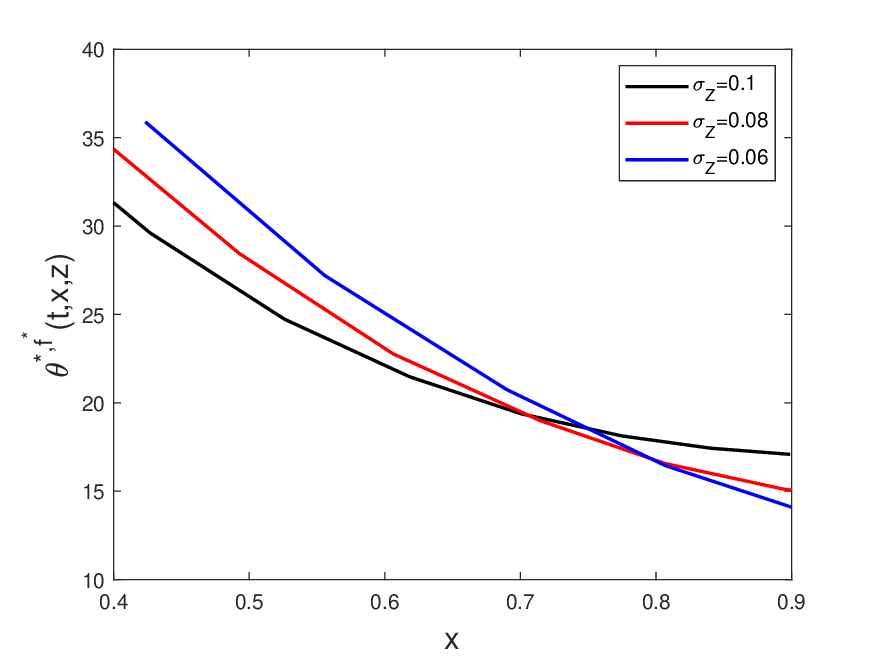}
\caption{ The portfolio feedback function $x\mapsto \theta^{*,f^*}(t,x,z)$. The parameters are set to be $(x_0,z_0)=(2.02,20)$, $(t,z)=(0.5,20)$.}\label{fig:f8}
\end{figure}

\section{Approximate Nash Equilibrium of $n$-Player Game}\label{sec:approximate}
The goal of this section is to show that the mean field equilibrium obtained in Theorem \ref{thm:MF-equilibrium} can help us to construct an approximate Nash equilibrium in the finite $n$-player game with heterogeneous agents when the type vector satisfies certain convergence condition and $n$ is sufficiently large.

Recall the objective functional of agent $i$ that, for 
\begin{align}\label{eq:n-player-game}
J^{i,n}(t,x_0,z_0;\theta^{i,n},{\bm \theta}^{-i,n}):=\Ex\left[  -\int_t^{T} e^{-\rho (s-t)}dL_s^{i,n}\big|X_t^{i,n}=x_0,Z_t^{i,n}=z_0 \right],
\end{align}
where, for $i=1,\ldots,n$, the vector of policies ${\bm \theta}^{-i,n}$ is defined by ${\bm \theta}^{-i,n}=(\theta^{1,n},\theta^{2,n},\ldots$, $\theta^{i-1,n},\theta^{i+1,n},\ldots,\theta^{n,n})$. The dynamics of underlying state process $X^{i,n}=(X_t^{i,n})_{s\in[t,T]}$ with reflection at zero is given by, for 
\begin{align}\label{state-X-i-n}
dX_s^{i,n}&=\left(\theta_s^{i,n} \mu^{i,n} - \frac{\lambda^{i,n}}{n}\sum_{j=1}^n \mu^{j,n}\theta^{j,n}_s-(1-\lambda^{i,n})\mu_Z^{i,n} Z_s^{i,n}\right)ds\\
&\quad+\left(\theta_s^{i,n}\sigma^{i,n}-\frac{\lambda^{i,n}}{n}\sum_{j=1}^n \sigma^{j,n}\theta^{j,n}_s-(1-\lambda^{i,n})\sigma_Z^{i,n} Z_s^{i,n} \right)dW_s^i+dL_s^{i,n},\nonumber
\end{align}
while the dynamics of underlying state process $Z^{i,n}=(Z_s^{i,n})_{s\in[t,T]}$ follows that, for $s\in[t,T]$,
\begin{align}\label{eq:Zt-nhihn}
d Z_s^{i,n}=\mu_Z^{i,n} Z_s^{i,n} d s+\sigma_Z^{i,n} Z_s^{i,n} d W_s^{i}.
\end{align}
Let $C_0>0$ be a constant, the admissible control set is defined as follows, for $(t,x_0,z_0)\in[0,T]\times\R_+^2$, 
\begin{align*}
\mathbb{U}_t(x_0,z_0):=\left\{\theta=(\theta_s)_{s\in[t,T]};~
\theta\in \mathbb{U}_t~\text{ with}~\sup_{s\in[t,T]}\Ex[|\theta_s|^2]\leq C_0(1+x_0^2+z_0^2). \right\}.
\end{align*}
 Denote by $\mathbb{U}_t^{n}(x_0,z_0):=\prod_{i=1}^n \mathbb{U}_t(x_0,z_0)$. Then, the definition of an approximate Nash equilibrium for the $n$-player game is given below:
\begin{definition}[Approximate Nash Equilibrium (ANE)] 
Let $(t,x_0,z_0)\in [0,T]\times \R_+^2$. An admissible strategy ${\bm \theta}^{*,n}=(\theta^{*,1,n},\ldots,\theta^{*,n,n}) \in \mathbb{U}_t^{n}(x_0,z_0)$ is called an $\epsilon_n$-Nash equilibrium to the $n$-player game problem \eqref{eq:n-player-game}-\eqref{eq:Zt-nhihn} if it holds that
\begin{align*}
\sup _{\theta^{i,n} \in \mathbb{U}_t(x_0,z_0)} J^{i,n}\left(t,x_0,z_0;\theta^{i,n},{\bm \theta}^{*,-i,n}\right) \leq J^{i,n}\left(t,x_0,z_0;{\bm \theta}^{*,n}\right)+\epsilon_n, \quad \forall i=1, \ldots, n,
\end{align*}
where $\epsilon_n\to0$ as $n\to\infty$.
\end{definition}
To construct the ANE, we introduce the following auxiliary stochastic control problem of agent $i\in\{1,2,\ldots,n\}$, which is described as, for 
\begin{align}\label{eq:n-player-game-auxiliary}
\bar{u}^{i,n}(t,x_0,z_0)&:=\sup_{\theta^{i,n}\in\mathbb{U}_t(x_0,z_0)}\bar{J}^{i,n}(t,x_0,z_0;\theta^{i,n})\nonumber\\
&:=\sup_{\theta^{i,n}\in\mathbb{U}_t(x_0,z_0)}\Ex\left[  -\int_t^{T} e^{-\rho (s-t)}d\bar{L}_s^{i,n}\big|\bar{X}_t^{i,n}=x_0,Z_t^{i,n}=z_0 \right],
\end{align}
where, the the underlying state process $\bar{X}^{i,n}=(\bar{X}_s^{i,n})_{s\in[t,T]}$ with reflection at zero is given by, for $s\in[t,T]$,
\begin{align}\label{state-X-i}
d\bar{X}_s^{i,n}&=\left(\theta_s^{i,n} \mu^{i,n} -\lambda^{i,n} f^*(s)-(1-\lambda^{i,n})\mu_Z^{i,n} Z_s^{i,n}\right)ds+\left(\theta_s^{i,n}\sigma^{i,n}-(1-\lambda^{i,n})\sigma_Z^{i,n} Z_s^{i,n}\right)dW_s^i\nonumber\\
&\quad+d\bar{L}_s^{i,n},
\end{align}
while the state process $Z^{i,n}=(Z_s^{i,n})_{s\in[t,T]}$ is given by \eqref{eq:Zt-nhihn}.  

It follows from Corollary \ref{coro:optimal-portfolio-dual} that, the optimal portfolio strategy $\theta^{*,i,n}=(\theta^{*,i,n}_s)_{s\in[t,T]}$ admits the representation given by $\theta^{*,i,n}_s=\theta^{*,i,n}(s,R_s^{i,n,t,r},Z_s^{i,n,t,z_0})$ with $r=r(t,x_0,z_0):=-\ln  \bar{u}^{i,n}_x(t,x_0,z_0)$ for $(s,x_0,z_0)\in {\cal O}_T^{i,n}$. Here, the function $\theta^{*,i,n}(s,r,z_0)$ is given by, for $(s,r,z_0)\in[0,T]\times\R\times\R_+$,
{\footnotesize
\begin{align}\label{optimal-portfolio-auxiliary}
&\theta^{*,i,n}(s,r,z_0)\\
&=\begin{cases}
\displaystyle \frac{\lambda^{i,n} \mu^{i,n}}{(\sigma^{i,n})^2} \int_{s}^{T}\int_{-\infty}^{r} e^{-\rho (m-s)+\ell}\phi^{i,n}(m-s,x_0,r)f^*(m) d\ell dm+\frac{(1-\lambda^{i,n})\eta^i\sigma_Z^{i,n}}{\sigma^{i,n}}  \varphi^{i,n}(s,r,z_0)+\frac{(1-\lambda^{i,n})\sigma_Z^{i,n}}{\sigma^{i,n}}z_0\nonumber\\[1em]
\displaystyle \quad+\frac{(1-\lambda^{i,n})\eta^{i,n} \mu^{i,n}}{(\sigma^{i,n})^2}z_0\int_{t}^{T}\int_{-\infty}^{r} \exp\left(-(\rho-\kappa^{i,n})(m-s)+\left(1-\frac{\sigma^{i,n}\sigma_Z^{i,n}}{\mu^{i,n}}\right)\ell\right)\phi^{i,n}(m-s,\ell,r)d\ell dm,\\[1em]
\displaystyle\qquad\qquad\qquad\qquad\qquad\qquad\qquad \qquad\qquad\qquad\qquad\qquad\qquad\qquad\qquad\qquad\qquad\qquad\qquad(s,x_0,z_0)\in\mathcal{O}_T^{i,n},\\[1em]
\displaystyle \frac{(1-\lambda^{i,n})\sigma_Z^{i,n}}{\sigma^{i,n}}e^{\eta^{i,n} (T-t)}z_0,\qquad\qquad\qquad\qquad\qquad\qquad\qquad\qquad\qquad\qquad\qquad \quad(s,x_0,z_0)\in(\mathcal{O}_T^{i,n})^c \cap \overline{\mathcal{D}}_T,
\end{cases}
\end{align}}where, the process $R^{i,n,t,r}=(R^{i,n,t,r})_{s\in[t,T]}$, the functions $\varphi^{i,n}(\cdot), \phi^{i,n}(\cdot)$, the parameters $\eta^{i,n},\kappa^{i,n}$ and the region $\mathcal{O}_T^{i,n}$ are all defined as in the mean field sense, with $(\mu, \sigma, \mu_Z, \sigma_Z, \lambda)$ and $W$ replaced by $(\mu^{i,n}, \sigma^{i,n}, \mu_Z^{i,n}, \sigma_Z^{i,n}, \lambda^{i,n})$ and $W^i$.  

It follows from the similar proof of Theorem \ref{thm:verification} that $\sup_{s\in[t,T]}\mathbb{E}[|\theta_s^{*,i,n}|^2] \le C_{i,n}(1+x_0^2+z_0^2)$, where $C_{i,n}=C(\mu^{i,n},\sigma^{i,n},\mu_Z^{i,n},\sigma_Z^{i,n},\lambda^{i,n})$ with $C(\cdot)$ being a continuous function. Since $(\mu^{i,n},\sigma^{i,n},\mu_Z^{i,n},\sigma_Z^{i,n},\lambda^{i,n})$ converge uniformly to $(\mu,\sigma,\mu_Z,\sigma_Z,\lambda)$ as $n\to\infty$, we can choose a constant $C_0>0$ independent of $i$ and $n$ sufficiently large such that $\theta^{*,i,n}\in\mathbb{U}_t(x_0,z_0)$ for all $n\ge 1$ and $i=1,\ldots,n$.  

To construct an ANE, we require the following auxiliary result.
\begin{lemma}\label{lem:approximate}
Assume that $(\mu^{i,n}, \sigma^{i,n}, \mu_Z^{i,n}, \sigma_Z^{i,n}, \lambda^{i,n})$ converge uniformly to $(\mu, \sigma, \mu_Z, \sigma_Z, \lambda)$ as $n\to\infty$. For $(t,x_0,z_0)\in[0,T]\times\R_+^2$, let $\theta^{*,i,n}\in\mathbb{U}_t(x_0,z_0)$ be the optimal control of problem \eqref{eq:n-player-game-auxiliary} for $i=1,\ldots,n$, which is given by \eqref{optimal-portfolio-auxiliary}. For $i=1,\ldots n$ and $\theta^{i,n}\in \mathbb{U}_t(x_0,z_0)$, define the processes $X^{*,i,n}=(X_s^{*,i,n})_{s\in[t,T]}$ and $X^{*,-i,n}=(X_s^{*,-i,n})_{s\in[t,T]}$ as follows, for $s\in(t,T]$,
\begin{align}\label{X-i-optimal-1}
dX_s^{*,i,n}&=\left(\theta_s^{*,i,n} \mu^{i,n} - \frac{\lambda^{i,n}}{n}\sum_{j=1}^n \mu^{j,n}\theta^{*,j,n}_s-(1-\lambda^{i,n})\mu_Z^{i,n} Z_s^{i,n}\right)ds\\
&\quad+\left(\theta_s^{*,i,n}\sigma^{i,n}- \frac{\lambda^{i,n}}{n}\sum_{j=1}^n \sigma^{j,n}\theta^{*,j,n}_s-(1-\lambda^{i,n})\sigma_Z^{i,n} Z_s^{i,n} \right)dW_s^i+dL_s^{*,i,n},\nonumber
\end{align}
and 
\begin{align}\label{X-i-optimal-2}
dX_s^{*,-i,n}&=\left(\theta_s^{i,n} \mu^{i,n}\left(1-\frac{\lambda^{i,n} }{n}\right) -\frac{\lambda^{i,n} }{n}\sum_{j=1,\ldots,n, j\neq i}\mu^{j,n}\theta^{*,j,n}_s-(1-\lambda^{i,n})\mu_Z^{i,n} Z_s^{i,n}\right)ds\nonumber\\
&+\left(\theta_s^{i,n}\sigma^{i,n}\left(1-\frac{\lambda^{i,n} }{n}\right) -\frac{\lambda^{i,n} }{n}\sum_{j=1,\ldots,n, j\neq i}\sigma^{j,n}\theta^{*,j,n}_s-(1-\lambda^{i,n})\sigma_Z^{i,n} Z_s^{i,n} \right)dW_s^i\nonumber\\
&+dL_s^{*,-i,n}
\end{align}
with $X_t^{*,i,n}=X_t^{*,-i,n}=x_0$. Here, the process $Z^{i,n}=(Z_s^{i,n})_{s\in[t,T]}$ is given by \eqref{eq:Zt-nhihn} with $Z_t^i=z_0$ and the processes $L_s^{*,i,n}, L_s^{*,-i,n}$ for $s\in[t,T]$ are the local time processes of $X_s^{*,i,n},X_t^{*,-i,n}$ for $s\in[t,T]$, respectively. Then, we have
\begin{align}\label{eq:approximate-1}
\lim_{n\to \infty}\left(\sup_{s\in[t,T]}\Ex\left[\left|X_s^{*,i,n}-\bar{X}_s^{*,i,n}\right|\right]+\sup_{s\in[t,T]}\Ex\left[\left|L_s^{*,i,n}-\bar{L}_s^{*,i,n}\right|\right]\right)=0,
\end{align}
and
\begin{align}\label{eq:approximate-2}
\lim_{n\to \infty} \sup_{\theta^{i,n}\in \mathbb{U}_t(x_0,z_0)}\left(\sup_{s\in[t,T]}\Ex\left[\left|X_s^{*,-i,n}-\bar{X}_s^{i,n}\right|\right]+\sup_{s\in[t,T]}\Ex\left[\left|L_s^{*,-i,n}-\bar{L}_s^{i,n}\right|\right]\right)=0,
\end{align}
where, the pair of processes $(\bar{X}^{i,n},\bar{L}^{i,n})$ follows \eqref{state-X-i} under $\theta^{i,n}\in\mathbb{U}_t(x_0,z_0)$; while $(\bar{X}^{*,i,n},\bar{L}^{*,i,n})$ obeys \eqref{state-X-i} under $\theta^{*,i,n}\in\mathbb{U}_t(x_0,z_0)$.
\end{lemma}

The following theorem, which is the main result of this section, establishes an $\epsilon_n$-Nash equilibrium for the $n$-player game.
\begin{theorem}\label{thm:approximation-Nash}
Let the condition of Lemma~\ref{lem:approximate} hold. For any $(t,x_0,z_0)\in[0,T]\times\R_+^2$, let $\theta^{*,i,n}\in\mathbb{U}_t(x_0,z_0)$ be the optimal control of problem \eqref{eq:n-player-game-auxiliary} for $i=1,\ldots,n$, which is given by \eqref{optimal-portfolio-auxiliary}. Then, the policy ${\bm \theta}^{*,n}=(\theta^{*,i,n},\ldots,\theta^{*,n,n})$ is an $\epsilon_n$-Nash equilibrium to the $n$-player game with $\lim_{n\to \infty}\epsilon_n=0$.
\end{theorem}


\begin{proof}
For $i=1,\ldots,n$ and $\theta^i \in \mathbb{U}_t^{n}(x_0,z_0)$, we have 
\begin{align}\label{eq:Nash-ineq-1}
&J^{i,n}(t,x_0,z_0;\theta^{i,n},{\bm \theta}^{*,-i,n})-J^{i,n}(t,x_0,z_0;{\bm \theta}^{*,n})\nonumber\\
&\quad=J^{i,n}(t,x_0,z_0;\theta^{i,n},{\bm \theta}^{*,-i,n})- \sup_{\theta\in \mathbb{U}_t^{\rm r}(x_0,z_0)} \bar{J}^{i,n}(t,x_0,z_0;\theta)\nonumber\\
&\qquad+\sup_{\theta\in \mathbb{U}_t^{\rm r}(x_0,z_0)}\bar{J}^{i,n}(t,x_0,z_0;\theta)-J^{i,n}(t,x_0,z_0;{\bm \theta}^{*,n})\nonumber\\
&\quad \leq \left(J^{i,n}(t,x_0,z_0;\theta^{i,n},{\bm \theta}^{*,-i,n})-  \bar{J}^{i,n}(t,x_0,z_0;\theta^{i,n})\right)\nonumber\\
&\qquad+\left( \bar{J}^{i,n}(t,x_0,z_0;\theta^{*,i,n})-J^{i,n}(t,x_0,z_0;{\bm \theta}^{*,n})\right).
\end{align}
For the 1st term in last inequality of \eqref{eq:Nash-ineq-1}, integration by parts yields that
\begin{align}\label{eq:Nash-ineq-2}
&J^{i,n}(t,x_0,z_0;\theta^{i,n},{\bm \theta}^{*,-i,n})-  \bar{J}^{i,n}(t,x_0,z_0;\theta^{i,n})\nonumber\\
&\qquad=\Ex\left[  -\int_t^{T} e^{-\rho (s-t)}dL_s^{*,-i,n}\right]-\Ex\left[  -\int_t^{T} e^{-\rho (s-t)}d\bar{L}_s^{i,n}\right]\nonumber\\
&\qquad=-e^{-\rho (T-t)}\Ex\left[L_T^{*,-i,n}-\bar{L}_T^{i,n}\right]-\rho \int_t^{T} e^{-\rho (s-t)}\Ex\left[L_s^{*,-i,n}-\bar{L}_s^{i,n}\right]ds\nonumber\\
&\qquad\leq (1+\rho (T-t))\sup_{s\in[t,T]}\Ex\left[\left|L_s^{*,-i,n}-\bar{L}_s^{i,n}\right|\right].
\end{align}
In a similar fashion, for the 2nd term in last inequality of \eqref{eq:Nash-ineq-1}, we have 
\begin{align}\label{eq:Nash-ineq-3}
&\bar{J}^{i,n}(t,x_0,z_0;\theta^{*,i,n})-J^{i,n}(t,x_0,z_0;{\bm \theta}^{*,n})\leq (1+\rho (T-t))\sup_{s\in[t,T]}\Ex\left[\left|L_s^{*,i,n}-\bar{L}_s^{*,i,n}\right|\right].
\end{align}
It follows from  \eqref{eq:Nash-ineq-1}- \eqref{eq:Nash-ineq-3} and Lemma \ref{lem:approximate} that, as $n\to\infty$,
\begin{align*}
&\qquad\sup _{\theta \in \mathbb{U}_t(x_0,z_0)} J^{i,n}\left(t,x_0,z_0;\theta,{\bm \theta}^{*,-i,n}\right) -J^{i,n}\left(t,x_0,z_0;{\bm \theta}^{*,n}\right)\\
&\leq (1+\rho (T-t))\sup _{\theta \in \mathbb{U}_t(x_0,z_0)} \sup_{s\in[t,T]}\left\{\Ex\left[\left|L_s^{*,-i,n}-\bar{L}_s^{i,n}\right|\right]+\Ex\left[\left|L_s^{*,i,n}-\bar{L}_s^{*,i,n}\right|\right]\right\}\to 0,
\end{align*}
where, the sequence of positive constants $(\epsilon_n)_{n\geq1}$ is given by 
\begin{align*}
\epsilon_n:=(1+\rho (T-t))\sup _{\theta \in \mathbb{U}_t(x_0,z_0)} \sup_{s\in[t,T]}\left\{\Ex\left[\left|L_s^{*,-i,n}-\bar{L}_s^{i,n}\right|\right]+\Ex\left[\left|L_s^{*,i,n}-\bar{L}_s^{*,i,n}\right|\right]\right\}.
\end{align*}
Thus, we complete the proof of the theorem.
\end{proof}

\section{Conclusion}\label{sec:conclusion}
In this paper, we have studied a class of MFG problems of optimal tracking portfolio management, where each agent aims to minimize a novel risk measure—the expected largest shortfall relative to a benchmark that combines the population's average wealth with a market index. By developing a reflected dual method, we linearized the nonlinear HJB equation with Neumann boundary conditions, provided a probabilistic representation of the solution, and explicitly characterized the MFE, including a free boundary separating outperforming and underperforming regions. The MFE was further used to construct an $\epsilon_n$-Nash equilibrium for finite player games with $\epsilon_n\to0$ as the population size $n$ grows large.  Our work solves a new class of MFG problems with state reflections analytically and characterizes an MFE in this context. The convex dual method we develop, based on reflected state processes, enables a rigorous proof of the MFE's consistency condition through the analysis of the reflected dual process. As future work, we will focus on relaxing homogeneity by introducing multi‑population MFG or Graphon MFG.

\section{Proofs}\label{sec:proofs}
This section provides the detailed proofs of main results in previous sections.

We first show that $v(t,r,z)$ given by \eqref{eq:v} is a classical solution to the Neumann boundary problem \eqref{eq:HJB-v}.
\begin{lemma}\label{lem:derivative-v}
{\it Let $f\in\C([0,T])$ satisfy $f(t)> 0$ for all $t\in[0,T]$. Then, $v\in{\C}^{1,2,2}(\overline{\mathcal{D}}_T)$ and $v_r\in$ ${\C}^{1,2,2}(\overline{\mathcal{D}}_T)$. Moreover, for $(t, r, h) \in\overline{\mathcal{D}}_T$, we have that
\begin{align}
v_r(t, r,z)&=\lambda\mathbb{E}\left[\int_t^{\tau_r^t \wedge T} e^{-\rho s-R_s^{t, r}} f(s)ds\right]+(1-\lambda)\eta  \Ex\left[ \int_t^{\tau_r^t \wedge T}e^{-\rho s-R^{t,r}_s}Z_s^{t,z} ds\right], \label{eq:v-r}\\
v_{rr}(t,r,z)&=\lambda \int_{t}^{T}\int_{-\infty}^{r} e^{-\rho s-r+x}\phi(s-t,x,r)f(s)dx ds-\lambda  \Ex\left[\int_{t}^{\tau_r^t\wedge T} e^{-\rho s-R_s^{t,r}}f(s)ds\right]\nonumber\\
&~+(1-\lambda)\eta z\int_{t}^{T}\int_{-\infty}^{r} \exp\Bigg(-\rho s-r+\left(1-\frac{\sigma\sigma_Z}{\mu}\right)x+\kappa(s-t)\Bigg) \phi(s-t,x,r)dx ds\nonumber\\
&~-(1-\lambda)\eta  \Ex\left[ \int_t^{\tau_r^t \wedge T}e^{-\rho s-R^{t,r}_s}Z_s^{t,z} ds\right],\label{eq:v-rr}
\end{align}
where $\kappa:=\mu_Z-\frac{\sigma_Z^2}{2}-\frac{\mu\sigma_Z}{2\sigma}+\frac{\sigma\sigma_Z\rho}{\mu}$, the stopping time $\tau_r^t$ is defined by
\begin{align*}
\tau_r^t:=\inf \left\{s \geq t;-\frac{\mu}{\sigma}\left(W_s-W_t\right)-\left(\frac{\mu^2}{2\sigma^2}-\rho\right)(s-t)=r\right\},
\end{align*}
and we recall that the function $\phi(s,x,y)$ is given by \eqref{eq:phi}.
}
\end{lemma}

\begin{proof}
For $(t,r)\in [0,T]\times \R_+$, let us define 
\begin{align}
h(t,r)&:=-(1-\lambda)\eta \Ex\Bigg[ \int_t^T\exp\Bigg(-\rho s-R^{t,r}_s+\left(\mu_Z-\frac{\sigma_Z^2}{2}\right)(s-t)+\sigma_Z(W_s-W_t)\Bigg) ds\Bigg],\label{eq:h}\\
l(t,r)&:=-\lambda\Ex\left[ \int_t^Te^{-\rho s-R^{t,r}_s}f(s) ds\right].\label{eq:l}
\end{align}
Then, $v(t,r,z)=zh(t,r)+l(t,r)$ for $(t,r,z)\in \overline{\mathcal{D}}_T$. By a similar calculation as in the proof of Proposition  4.1 in \cite{BoLiaoYu21} and in the proof of Lemma 3.4 in \cite{BHY24a}, we can obtain 
\begin{align}
l_r(t, r)&=\lambda\mathbb{E}\left[\int_t^{\tau_r^t \wedge T} e^{-\rho s-R_s^{t, r}} f(s)ds\right].\label{eq:l-r-1}\\
zh_r(t, r)&=(1-\lambda)\eta z\Ex\Bigg[ \int_t^{\tau_r^t \wedge T}\exp\Bigg(-\rho s-R^{t,r}_s+\left(\mu_Z-\frac{\sigma_Z^2}{2}\right)(s-t)+\sigma_Z(W_s-W_t)\Bigg) ds\Bigg].\nonumber\\
l_{rr}(t,r)&=\lambda  \int_{t}^{T}\int_{-\infty}^{r} e^{-\rho s-r+x}\phi(s-t,x,r)f(s)dx ds-\lambda  \Ex\left[\int_{t}^{\tau_r\wedge T} e^{-\rho s-R_s^{t,r}}f(s)ds\right].\label{eq:l-r-2}\\
zh_{rr}(t,r)&=(1-\lambda)\eta z\int_{t}^{T}\int_{-\infty}^{r} \exp\Bigg(-\rho s-r+\left(1-\frac{\sigma\sigma_Z}{\mu}\right)x+\kappa(s-t)\Bigg) \phi(s-t,x,r)dx ds\nonumber\\
&-(1-\lambda)\eta z \Ex\Bigg[ \int_t^{\tau_r^t \wedge T}\exp\Bigg(-\rho s-R^{t,r}_s+\left(\mu_Z-\frac{\sigma_Z^2}{2}\right)(s-t)+\sigma_Z(W_s-W_t)\Bigg) ds\Bigg],\\
l_{rrr}(t,r)&=-2\lambda  \int_{t}^{T}\int_{-\infty}^{r} e^{-\rho s-r+x}\phi(s-t,x,r)f(s)dx ds+\lambda  \Ex\left[\int_{t}^{\tau_r\wedge T} e^{-\rho s-R_s^{t,r}}f(s)ds\right]\label{eq:l-r-3}\\
&+\lambda  \int_{t}^{T} e^{-\rho s}\phi(s-t,r,r)f(s) ds+\lambda  \int_{t}^{T}\int_{-\infty}^{r} e^{-\rho s-r+x}\phi_r(s-t,x,r)f(s)dx ds,\nonumber\\
h_{rrr}(t,r)&=-2(1-\lambda)\eta\int_{t}^{T}\int_{-\infty}^{r} \exp\Bigg(-\rho s-r+\left(1-\frac{\sigma\sigma_Z}{\mu}\right)x+\kappa(s-t)\Bigg) \phi(s-t,x,r)dx ds\nonumber\\
&\quad+(1-\lambda)\eta \Ex\Bigg[ \int_t^{\tau_r^t \wedge T}\exp\Bigg(-\rho s-R^{t,r}_s+\left(\mu_Z-\frac{\sigma_Z^2}{2}\right)(s-t)+\sigma_Z(W_s-W_t)\Bigg) ds\Bigg]\nonumber\\
&\quad+(1-\lambda)\eta\int_{t}^{T} \exp\left(-\rho s-\frac{\sigma\sigma_Z}{\mu}r+\kappa(s-t)\right)\times \phi(s-t,r,r)dx ds\nonumber\\
&\quad-(1-\lambda)\eta\int_{t}^{T}\int_{-\infty}^{r} \exp\Bigg(-\rho s-r+\left(1-\frac{\sigma\sigma_Z}{\mu}\right)x+\kappa(s-t)\Bigg) \phi_r(s-t,x,r)dx ds.\label{eq:h-r-3}
\end{align}
From \eqref{eq:l-r-1}-\eqref{eq:h-r-3}, we deduce \eqref{eq:v-r}-\eqref{eq:v-rr} and that $v_{rrr}=l_{rrr}+zh_{rrr}$ is continuous in $(t,r,z)$. We next derive the expression of $v_z,v_{zz},v_{rz}$ and $v_{rrz}$. Note that  $v(t,r,z)=zh(t,r)+l(t,r)$. Then, for $(t,r,z)\in\overline{\mathcal{D}}_T$,
\begin{align}
&v_z(t,r,z)=h(t,r),\quad v_{zz}(t,r,z)=0,\label{eq:v-z}\\
&v_{rz}(t,r,z)=h_r(t,r,z)\quad v_{rrz}(t,r,z)=h_{rr}(t,r,z).\label{eq:v-z2}
\end{align}
Finally, we derive the expression of $v_t$ and $v_{rt}$. For any $\delta\in[-t,T-t]$, it holds that
\begin{align*}
l(t,r)&=-\lambda \Ex\left[\int_{t}^T e^{-\rho s-R_s^{t,r}}f(s)ds\right]=-\lambda \Ex\left[\int_{t}^{T}e^{-\rho s-R_{s+\delta}^{t+\delta,r}}f(s)ds\right].
\end{align*}
Then, we get
\begin{align*}
&\lim_{\delta\to 0}\frac{l(t+\delta,r)-l(t,r)}{\delta}\\
&=-\lambda \lim_{\delta\to 0}\frac{1}{\delta}\left\{ \Ex\left[\int_{t+\delta}^{T} e^{-\rho s-R_{s}^{t+\delta,r}}f(s)ds\right]-\Ex\left[ \int_t^Te^{-\rho s-R^{t+\delta,r}_{s+\delta}}f(s)ds\right] \right\}\\
&=\lambda\lim_{\delta\to 0}\frac{1}{\delta}\Ex\left[\int_t^{t+\delta}e^{-\rho s-R_{s}^{t+\delta,r}}f(s)ds\right]-\lambda  \int_{t}^{T} e^{-\rho s}f(s)\lim_{\delta\to 0}\left\{\frac{1}{\delta}\Ex\left[e^{-R_{s}^{t+\delta,r}}-e^{-R^{t+\delta,r}_{s+\delta}}\right]\right\}ds\\
&=\lambda e^{-\rho t-r}f(t)-\lambda  \int_{t}^{T} e^{-\rho s}f(s)\lim_{\delta\to 0}\left\{\frac{1}{\delta}\Ex\left[e^{-R_{s}^{t+\delta,r}}-e^{-R^{t+\delta,r}_{s+\delta}}\right]\right\}ds.
\end{align*}
By applying It\^o's rule and taking expectation, we have $\Ex[e^{-R_{s}^{t+\delta,r}}-e^{-R^{t+\delta,r}_{s+\delta}}]=-\rho \delta-\Ex[\int_{s+\delta}^s  dL_{\ell}^{t+\delta,r}]$. Using the dominated convergence theorem and Proposition~2.5 in \cite{Abraham2000}, we arrive at
\begin{align}\label{eq:l-t-1}
l_t(t,r)&=\lambda e^{-\rho t-r}f(t)+\lambda  \int_{t}^{T} e^{-\rho s}f(s)(\rho-p(t,r;s,0))ds,
\end{align}
where $p(t,r;s,x)$ is the conditional density function of the reflected drifted Brownian process $R^{t,r}$ (c.f. \cite{Vee2004}). Similar to the calculation above, we can also get
\begin{align}\label{eq:h-t-1}
zh_t(t,r)
=(1-\lambda)\eta \Bigg\{&\Ex\left[ \int_{t}^{T} e^{-\rho s-R^{t,r}_{s}} Z_{s}^{t,z} ds\right]+\Ex\left[  e^{-\rho T-R^{t,r}_{T}} Z_{T}^{t,z} \right]\Bigg\}.
\end{align}
We also have from \eqref{eq:v-r} that
\begin{align}\label{eq:v-t-2}
&v_{rt}(t,r)=-\lambda \int_t^T \int_0^r \int_{-\infty}^y e^{-\rho s-r+x}\phi_s(s-t,x,y)dsdyds\\
&\quad-(1-\lambda)\eta z\int_{t}^{T}\int_0^r \int_{-\infty}^{y} \exp\left(-\rho s-r+\left(1-\frac{\sigma\sigma_Z}{\mu}\right)x+\kappa(s-t)\right) \phi_s(s-t,x,y)dx dyds.\nonumber
\end{align}
By \eqref{eq:l-r-1}, \eqref{eq:l-r-2}, \eqref{eq:v-z} \eqref{eq:l-t-1} and \eqref{eq:h-t-1}, we have $v \in {\C}^{1,2}(\overline{\mathcal{D}}_T)$. Using \eqref{eq:l-r-2}, \eqref{eq:l-r-3}, \eqref{eq:h-r-3}, \eqref{eq:v-z} and \eqref{eq:v-t-2}, we further have $v_r \in {\C}^{1,2}(\overline{\mathcal{D}}_T)$.
\end{proof}

Building upon Lemma \ref{lem:derivative-v}, we have the following result:
\begin{proposition}\label{prop:v}
{\it Let $f\in {\C}([0,T])$ satisfy $f(t)> 0$ for all $t\in[0,T]$. Then, the function $v$ defined by \eqref{eq:v} is the unique classical solution to the Neumann boundary problem \eqref{eq:HJB-v} satisfying $|v(t,r,z)|\leq C(1+z)$ for some constant $C>0$.} 
\end{proposition}

{\it Proof of Proposition \ref{thm:dual-u}:}\quad It follows from Proposition \ref{prop:v} that  $\hat{u}(t,y,z)$ is the  unique classical solution of the dual equation \eqref{eq:dual-u} satisfying $|\hat{v}(t,y,z)|\leq K(1+z)$ for some $K>0$ and $\hat{u}_y \in {\C}^{1,2}([0,T]\times (0,1]\times\R_+)$. Furthermore, the strict convexity of $(0,1] \ni y \mapsto \hat{u}(t,y)$ for fixed $t \in[0, T) $ follows from the fact that 
\begin{align*}
&\hat{u}_{y y}=\frac{e^{\rho t}}{y^2}(v_{rr}+v_r)=\frac{\lambda e^{\rho t}}{y^2} \int_{t}^{T}\int_{-\infty}^{r} e^{-\rho s-r+x}\phi(s-t,x,r)f(s)dx ds\\
&\qquad+\frac{(1-\lambda)\eta e^{\rho t}}{y^2}\int_{t}^{T}\int_{-\infty}^{r} \exp\left(-\rho s-r+\left(1-\frac{\sigma\sigma_Z}{\mu}\right)x+\kappa(s-t)\right) \phi(s-t,x,r)dx ds>0.
\end{align*} 
Thus, we complete the proof of the proposition, \hfill$\blacksquare$\\

{\it Proof of Theorem \ref{thm:verification}:}\quad
The proof of the item  (i) is similar to that of Theorem 5.1 in \cite {BoLiaoYu21}, and hence we omit it here. Next, we show the item (ii). It is not difficult to see that
\begin{align}\label{eq:theta-1}
&\lim_{y\downarrow 0}z \hat{u}_{yz}(t,y,z)=-\lim_{y\downarrow 0}\frac{e^{\rho t}z}{y}v_{rz}(t,-\ln y,z)\nonumber\\
&\quad=-\lim_{r\to +\infty}e^{\rho t+r}zv_{rz}(t,r,z)\nonumber\\
&\quad=-(1-\lambda)\eta\mathbb{E}\left[\int_t^{T} e^{-\frac{\mu}{\sigma}\left(W_s-W_t\right)-\frac{\mu^2}{2\sigma^2}(s-t)}  Z_s^{t,z}ds\right]\nonumber\\
&\quad=-(1-\lambda)\left(e^{\eta (T-t)}-1\right)z.
\end{align}
In view of $\mu>0$, \eqref{eq:v-rr} and $\hat{u}_{yy}\geq 0$ from Proposition \ref{thm:dual-u}, we have
\begin{align}\label{eq:estimate-theta-1}
&0\leq\liminf_{y\downarrow 0} y\hat{u}_{yy}(t,y,z)\leq \limsup_{y\downarrow 0} y\hat{u}_{yy}(t,y,z)=\limsup_{r\to +\infty} e^{\rho t+r}(v_{rr}+v_r)(t,r,z).
\end{align}
We obtain from Lemma \ref{lem:derivative-v} that
\begin{align}\label{eq:estimate-theta-2}
&e^{r}(v_{rr}+v_r)(t,r,z)\nonumber\\
&\quad=\lambda \int_{t}^{T}\int_{-\infty}^{r} e^{-\rho s+x}\phi(s-t,x,r)f(s)dx ds\\
&\qquad +(1-\lambda)\eta z\int_{t}^{T}\int_{-\infty}^{r} \exp\left(-\rho s+\left(1-\frac{\sigma\sigma_Z}{\mu}\right)x+\kappa(s-t)\right) \phi(s-t,x,r)dx ds.\nonumber
\end{align}
In the sequel, let $C>0$ be a generic constant independent of $(t,x,z)$, which may differ from line to line. For the case $r>2$, it holds that $(r+y)^2-1>0$ for all $y\geq0$. As $f\in\C([0,T])$, it follows that, for all $(t,r)\in[0,T]\times (2,\infty)$,
\begin{align}\label{eq:estimate-theta-3}
&\int_{t}^{T}\int_{-\infty}^{r} e^{-\rho s+x}\phi(s-t,x,r)f(s)dx\leq C \int_{t}^{T}\int_{-\infty}^{r} e^{-\rho s+x}\phi(s-t,x,r)dx ds\nonumber\\
&\overset{y=r-x}{=}C\int_{t}^{T}\int_{0}^{\infty} e^{-\rho s+r-y}\frac{2( r+y)}{\hat{\sigma}^2\sqrt{2\hat{\sigma} \pi (s-t)^3}}\exp\left(\frac{\hat{\mu}}{\hat{\sigma}}(r-y)-\frac{1}{2} \hat{\mu}^2 (s-t)-\frac{( r+y)^2}{2\hat{\sigma}^2 (s-t)}\right)dyds\nonumber\\
&=C\int_{t}^{T}\frac{2\exp\left(-\frac{1}{2\hat{\sigma}^2 (s-t)}\right)}{\hat{\sigma}^2\sqrt{2\hat{\sigma} \pi (s-t)^3}}e^{-\rho s}ds\int_{0}^{\infty} ( r+y)\exp\left(\frac{\hat{\mu}}{\hat{\sigma}} (r-y)+r-\frac{( r+y)^2-1}{2\hat{\sigma} ^2(T-t)}\right)dy\nonumber\\
&\leq C r\exp\left(\left(\frac{\hat{\mu}}{\hat{\sigma}}+1\right)r-\frac{r^2-1}{2\hat{\sigma}^2 (T-t)}\right)\int_{0}^{\infty} ( 1+y)\exp\left(-\frac{y^2}{2\hat{\sigma}^2 (T-t)}-\frac{\hat{\mu}}{\hat{\sigma}}y\right)dy\nonumber\\
&\leq Cr\exp\left(\frac{\hat{\mu}}{\hat{\sigma}}r+r-\frac{r^2-1}{2\hat{\sigma}^2 (T-t)}\right).
\end{align}
Since the mapping $r\to r\exp(\frac{\hat{\mu}}{\hat{\sigma}} r+r-\frac{r^2-1}{2\hat{\sigma}^2 (T-t)})$ is continuous, and
$\lim_{r\to\infty}r\exp(\frac{\hat{\mu}}{\hat{\sigma}}r +r-\frac{r^2-1}{2\hat{\sigma}^2 (T-t)})=0$, we have 
\begin{align}\label{eq:estimate-theta-4}
&\limsup_{r\to+\infty}\int_{t}^{T}\int_{-\infty}^{r} e^{-\rho s+x}\phi(s-t,x,r)f(s)dx\leq 0.
\end{align}
As the mapping $r\to r\exp(\frac{\hat{\mu}}{\hat{\sigma}} r+r-\frac{r^2-1}{2\hat{\sigma}^2 (T-t)})$ is continuous, and
$\lim_{r\to\infty}r\exp(\frac{\hat{\mu}}{\hat{\sigma}}r +r-\frac{r^2-1}{2\hat{\sigma}^2 (T-t)})=0$, we have 
\begin{align}\label{eq:estimate-theta-6}
&\limsup_{r\to+\infty}\int_{t}^{T}\int_{-\infty}^{r} e^{-\rho s+x}\phi(s-t,x,r)f(s)dx\leq 0.
\end{align}
In a similar fashion to derive \eqref{eq:estimate-theta-3}, we can also get that
\begin{align}\label{eq:estimate-theta-5}
&\limsup_{r\to+\infty}\int_{t}^{T}\int_{-\infty}^{r} \exp\left(-\rho s+\left(1-\frac{\sigma\sigma_Z}{\mu}\right)x+\kappa(s-t)\right) \phi(s-t,x,r)dx ds\leq 0.
\end{align}
Therefore, in view of \eqref{eq:estimate-theta-1} \eqref{eq:estimate-theta-2}, \eqref{eq:estimate-theta-4} and \eqref{eq:estimate-theta-5}, we deduce that
\begin{align}\label{eq:theta-2}
\lim_{y\downarrow 0} y\hat{u}_{yy}(t,y,z)=0.
\end{align}
Together with \eqref{eq:theta-1}, the  feedback control function $\theta^*(t,x,z)$ give by \eqref{eq:optimal-theta} is well-defined.
Moreover, we can see that the function $(x,z) \to \theta^*(t,x,z)$ is locally Lipschitz continuous and satisfies the liner growth with respect to $z$.
The SDE \eqref{eq:optimal-SDE} therefore admits a unique strong solution. Consequently, it is straightforward to verify that $\theta^*\in\mathbb{U}$.

 For any $\theta\in \mathbb{U}$, let $X=(X_t)_{t\in[ 0,T]}$ be the resulting state process. For $n\in \mathbb{N}$ and $1/n<T$, define the stopping time $\tau_n$ by $\tau_n:=\left(T-\frac{1}{n}\right)\wedge\inf \left\{s\geq 0: X_s \wedge Z_s\geq n\right\}$. By applying It\^o's lemma to $e^{-\rho \tau_n}u(\tau_n,X_{\tau_n},Z_{\tau_n})$ and  taking the expectation on both sides, we deduce
\begin{align}\label{eq:value-ineq}
&\Ex\left[e^{-\rho (\tau_n-t)}u(\tau_n,X_{\tau_n},Z_{\tau _n})\right]+\mathbb{E}\left[-\int_t^{\tau_n}e^{-\rho (s-t)}dL_s^X\right]\leq u(t,x,z),
\end{align}
and the equality holds when $\theta=\theta^*$. In view of Proposition \ref{thm:dual-u} and \eqref{eq:sol-u}, we know that $u(t,x,z)\leq 0$ on $\overline{\mathcal{D}}_T$ and $u(t,x,z)=\inf_{y\in(0,1]}\{\hat{u}(t,y,z)+yx\}\geq \inf_{y\in(0,1]}\{\hat{u}(t,y,z)\}\geq -C(1+z)$ on $\mathcal{O}_T$ for some $C>0$. This yields that $|u(t,x,z)|\leq C(1+z)$, for all $(t,x,z)\in\overline{\mathcal{D}}_T$.  Letting $n \rightarrow \infty$ on both sides of \eqref{eq:value-ineq}, by using the dominated convergence theorem and the monotone convergence theorem, we get the desired result. \hfill$\blacksquare$


By using Lemma \ref{lem:derivative-v} and Theorem \ref{thm:verification}, we can directly get Lemma  \ref{lem:u_x}. Then, we here provide the proofs of Lemma \ref{lem:dual-process}, Lemma \ref{lem:fixpoint-r} and Lemma \ref{lem:continuity-f}.

{\it Proof of Lemma \ref{lem:dual-process}:}\quad 
We first consider the case where $(0,x,z)\in {\cal O}_T$ and ${\cal O}_T\neq\varnothing$. In view of Lemma \ref{lem:u_x}, the partial derivative of the value function $u$ with respect to $x$ satisfies that  $u_x\in\C^{1,2}({\cal O}_T)$.  For any $t\in [0,T)$ and $\varepsilon\in(0,1)$, by applying It\^o's formula to $Y_{t\wedge (1-\varepsilon)\tau_0}=u_x({t\wedge (1-\varepsilon)\tau_0},X_{t\wedge (1-\varepsilon)\tau_0}^*,Z_{t\wedge (1-\varepsilon)\tau_0})$, we obtain
\begin{align}\label{eq:ito-Y-1}
&Y_{t\wedge (1-\varepsilon)\tau_0}=Y_0+\int_0^{{t\wedge (1-\varepsilon)\tau_0}} {\cal L}^{\theta_s^*} u_x(s,X_{s}^*)ds+\int_0^{{t\wedge (1-\varepsilon)\tau_0}}  u_{xx}(s,X_{s}^*)dL_s^X\nonumber\\
&\quad+\int_0^{{t\wedge (1-\varepsilon)\tau_0}} ((\theta_s^*\sigma-(1-\lambda)\sigma_Z Z_s )u_{xx}(s,X_{s}^*,Z_s)+\sigma_Z Z_s u_{xz}(s,X_{s}^*,Z_s) ) dW_s,
\end{align}
where  the operator ${\cal L}^{\theta}$ with $\theta\in \R$  acting on $g\in C^{1,2}(\mathcal{D}_T)$  is defined by
\begin{align*}
{\cal L}^{\theta}g&:=g_t+g_x\mu\theta+\frac{1}{2}g_{xx}(\sigma\theta-(1-\lambda)\sigma_Zz)^2+\frac{1}{2}\sigma_Z^2z^2g_{zz}\\
&\quad+g_{xz}(\sigma\theta-(1-\lambda)\sigma_Zz)\sigma_Z z+\mu_Z zg_z -((1-\lambda)\mu_Zz+\lambda f(t)) g_x.
\end{align*}
 We can obtain from \eqref{eq:optimal-theta} that
\begin{align}\label{eq:ito-Y-2}
& (\theta_s^*\sigma-(1-\lambda)\sigma_Z Z_s )u_{xx}(s,X_{s}^*,Z_s)+\sigma_Z Z_s u_{xz}(s,X_{s}^*,Z_s)\nonumber\\
&\quad =- \frac{\mu}{\sigma} u_{x}(s,X_{s}^*,Z_s)=- \frac{\mu}{\sigma}Y_s.
\end{align}
It follows from Theorem \ref{thm:verification} that $({\cal L}^{\theta_s^*} -\rho)u(t,x,z)=0$ for $(t,x,z)\in  {\cal O}_T$. Taking the derivative of $x$ on both sides of the equation, we have
\begin{align}\label{eq:ito-Y-3}
&\int_0^{{t\wedge (1-\varepsilon)\tau_0}}  {\cal L}^{\theta_s^*} u_x(s,X_{s}^*,Z_s)ds=\int_0^{{t\wedge (1-\varepsilon)\tau_0}} \rho  u_x(s,X_{s}^*,Z_s)ds=\int_0^{{t\wedge (1-\varepsilon)\tau_0}}  \rho Y_sds.
\end{align}
Denote by $L_t^Y:=-\int_0^{t}  u_{xx}(s,X_{s}^*,Z_s)dL_s^X$ for $t\in[0,T]$. Consequently, it follows from \eqref{eq:ito-Y-1}-\eqref{eq:ito-Y-3} and the arbitrariness of $\varepsilon\in(0,1)$ that
\begin{align*}
dY_t= \rho Y_tdt-\frac{\mu}{\sigma}Y_tdW_t-dL_t^Y,\quad t\in(0,\tau_0).
\end{align*}
Note that $u_x(t,x,z)\leq 1$, $u_x(t,0,z)= 1$ and $u_{xx}(t,x,z)\leq 0$ for all $(t,x,z)\in{\cal O}_T$. Then, the process $L=(L_t^Y)_{t\in[0,\tau_0)}$ is a continuous and non-decreasing process (with $L_0^{Y}=0$)  which increases on the time set $\{t\in[0,\tau_0);~Y_t=1\}$ only. This implies that  $Y$ taking values on $(0,1]$ is a reflected process and $L^Y$ is the local time process of $Y$. 

Next, we claim that $Y_t=e^{-R_t^{0,r}}$ on $[0,\tau_0)$, where  the reflected process $(R_t^{0,r})_{t\in [0,T]}$ is given by \eqref{eq:R} with $r=-\ln Y_0=-\ln v_x(0,x,z)$. Let $\tilde{Y}_t:=e^{-R_t^{0,r}}$ for $t\in[0,T]$. Then, by using It\^o's rule, $\tilde{Y}=(\tilde{Y}_t)_{t\in[0,T]}$ satisfies the following dynamics given by
\begin{align*}
d\tilde{Y}_t= \rho \tilde{Y}_tdt-\frac{\mu}{\sigma}\tilde{Y}_tdW_t-dL_t^{\tilde{Y}},
\end{align*}
where $L_t^{\tilde{Y}}=\int_0^t e^{-R_t^{0,r}} dL_s^{0,r}=\int_0^t  dL_s^{0,r}=L_t^{0,r}$. We can see that $L^{\tilde{Y}}=(L_t^{\tilde{Y}})_{t\in [0,T]}$ is a continuous and non-decreasing process that increases only on $\{t \in[0, T] ;~R_s^{0,r}=0\}=\{t \in[0, T] ;~\tilde{Y}_s=1\}$ with $L_0^{\tilde{Y}}=0$. This yields that $\tilde{Y}$ is a  the reflected process and $L^{\tilde{Y}}$ is the corresponding local time process. Noting that $\tilde{Y}$ and $Y$ satisfy the same reflected SDE, by the uniqueness of the strong solution, we deduce that $Y_t=\tilde{Y}_t=e^{-R_t^{0,r}}$ on $[0,\tau_0)$. Thus, it holds that, $\Px$-a.s.
\begin{align*}
\tau_0&=\inf\{t\in[0,T];~ X_t^*\geq x_0(t,Z_t)\}\wedge T\\
&=\inf\{t\in[0,T];~ v_x(t,X_t^*,Z_t)\leq0\}\wedge T\nonumber\\
&=\inf\{t\in[0,T];~ Y_t\leq 0\}\wedge T\\
&=\inf\{t\in[0,T];~ e^{-R_t^{0,r}}\leq 0\}\wedge T=T.
\end{align*}
Then, it is not hard to see that $Y$ satisfies the dynamics \eqref{eq:Yt} on $(0,T]$ with $Y_0=v_x(0,x,z)$.

We next let $(0,x,z)\in {\cal O}_T^c\cap \overline{\mathcal{D}}_T$. In this case, we have $x\leq x_0(0,z)$, and hence $\tau_0=0$, $\Px$-a.s., and the value function  $u(0,x,z)=0$. Introduce ${\cal M}:=\{\omega\in\Omega;~(t,X_t^*,Z_t)\in  {\cal O}_T^c\cap \overline{\mathcal{D}}_T\}$. Assume that there exists some $t_0\in(0,T]$ such that $\Px({\cal M})>0$. Then, it follows from the dynamical program principle and the fact $u(t,X_t^*,Z_t)<0$ on ${\cal M}$ that
\begin{align*}
0=u(0,x,z)&=\Ex\left[-\int_0^{t}e^{-\rho s}dL_s^{X^*}+e^{-\rho t}u(t,X_t^*,Z_t)\right]\\
&\leq\Ex\left[e^{-\rho t}u(t,X_t^*,Z_t){\bf 1}_{{\cal M}}\right]<0,
\end{align*}
which yields a contradiction. Thus, for any $t\in[0,T]$, we have $(t,X_t^*,Z_t)\in  {\cal O}_T^c\cap \overline{\mathcal{D}}_T$, $\Px$-a.s., which completes the proof of the lemma. \hfill$\blacksquare$\\

{\it Proof of Lemma \ref{lem:fixpoint-r}:}\quad 
In view of \eqref{eq:G} and \eqref{eq:H}, the functions $G(r,s,t)$ and $H(r,z,t)$ are non-negative and continuous.  Fix $(r,z)\in \R_+^2$. Choose $N>0$ with $0=t_0<t_1<t_2<\cdots<t_N=T$ and $t_{i+1}-t_i=T/N$ for $i=0,\ldots,N-1$ such that
\begin{align*}
M:=\sup_{i=0,...,N-1}  \sup_{t \in  [t_i,t_{i+1}]} \int_{t}^{t_{i+1}}G(r,s,t) ds<1.
\end{align*}
For $t\in [t_{N-1},t_N]$ and $f_1,f_2\in C([t_{N-1},t_N])$, we have 
\begin{align*}
\left|{\cal J}f_1(t)-{\cal J}f_2(t)\right|
&=\left|\int_t^T \left(f_1(s)-f_2(s)\right) G(r,s,t)ds\right|\\
&\leq \int_t^T \left|f_1(s)-f_2(s)\right|G(r,s,t)ds\\
&\leq \sup_{s\in[t,T]}|f_1(s)-f_2(s)| \int_t^T G(r,s,t)ds\\
&\leq M \|f_1-f_2\|_{[t_{N-1},t_N]},
\end{align*}
where $\|f\|_{[t_{N-1},t_N]}:=\sup_{t\in [t_{N-1},t_N]}|f(t)|$ for $f\in\C([t_{N-1},t_N])$ is the uniform norm of continuous function on  $[t_{N-1},t_N]$. This implies that ${\cal J}$ is a contraction map on $\C([t_{N-1},t_N])$, and hence has a unique fixed point $f^N$ on $[t_{N-1},t_N]$. Furthermore, if we start with a non-negative continuous function $f_0\in\C([t_{N-1},t_N])$, and consider the sequence $(f_n)_{n\geq 1}$ with $f_n:={\cal J}f_{n-1}$. It follows from  $G(r,s,t)\geq 0$ and $H(r,z,t)\geq 0$ that $f_n$ is also a non-negative function, for any $n\geq 1$. As $f_n$ converges to the fixed point $f^N$ as $n\to\infty$, we know $f^N(t)\geq 0$ on $[t_{N-1},t_N]$. On $[t_{N-2},t_{N-1}]$, let us consider the mapping ${\cal J}^{N-1}:\C([t_{N-2},t_{N-1}])\mapsto\C([t_{N-2},t_{N-1}])$ defined by, for $t\in[t_{N-2},t_{N-1}]$,
\begin{align*}
{\cal J}^{N-1}f(t)&:=\int_t^{t_{N-1}}f(s)G(r,s,t)ds+\int_{t_{N-1}}^{T} f^N(s)G(r,s,t)ds+H(r,z,t).
\end{align*}
Similarly, we can see that ${\cal J}^{N-1}$ is a contraction mapping on $\C([t_{N-2},t_{N-1}])$ , and hence has a unique fixed point $f^{N-1}(t)\geq 0$ on $[t_{N-2},t_{N-1}]$ with $f^{N-1}(t_{N-1})=f^N(t_{N-1})$. Repeating this procedure, we can get the non-negative continuous function sequence $\{f^i\}_{i=1}^N$. Define the function $f:[0,T]\mapsto \R_+$ with $f(t)=f^i(t)$ for $t\in[t_{i-1},t_i]$. Then, the non-negative continuous $f$ is the unique fixed point of the mapping ${\cal J}$. If there exists some $t_0\in[0,T)$ such that $f(t_0)\leq0$, we introduce $\hat{t}:=\sup\{t\in[0,T];~f(t)\leq 0\}$. As $f(T)>0,f(t_0)\leq 0$ and $f\in\C([0,T])$, we have $\hat{t}\in[t_0,T)$, and hence
\begin{align*}
f(\hat{t})=\int_{\hat{t}}^T f(s)G(r,s,t_0)ds+H(r,z,\hat{t})\leq 0,
\end{align*}
which contradicts with the facts that $f,G\geq 0$, $H>0$ on $(\hat{t},T]$. Thus, we complete the proof.
\hfill$\blacksquare$\\

{\it Proof of Lemma \ref{lem:continuity-f}:}\quad 
Fix $(r_1,t_1)\in\R_+\times [0,T]$. For any $\epsilon\in (0,1)$ and $(r_2,t_2)\in {\cal M}^{(r_1,t_1)}_{\epsilon}:=\{(r,t)\in \R_+\times [0,T];~|r-r_1|+|t-t_1|\leq \epsilon\}$, we have
\begin{align}\label{continuity} 
|f(t_1;r_1)-f(t_2;r_2)|&\leq |f(t_1;r_1)-f(t_2;r_1)|+|f(t_2;r_1)-f(t_2;r_2)|.
\end{align}
It follows from Lemma \ref{lem:fixpoint-r} that
\begin{align}\label{continuity-1}
|f(t_1;r_1)-f(t_2;r_1)|&\leq \left|\int_{t_1}^{t_2} f(s;r_1)G(r_1,s,t_1)ds\right|\left|H(r_1,z,t_1)-H(r_1,z,t_2)\right|\nonumber\\
&\leq \max_{s\in[t_1,t_2]}|f(s;r_1)G(r_1,s,t_1)||t_2-t_1|+\left|H(r_1,z,t_1)-H(r_1,z,t_2)\right|.
\end{align}
Utilizing the continuity of $(r,s,t)\mapsto G(r,s,t)$, $(r,t)\mapsto H(r,z,t)$ and  $s\mapsto f(s;r)$, we deduce from \eqref{continuity-1} that
\begin{align}\label{continuity-2}
\lim_{\epsilon \to0}|f(t_1;r_1)-f(t_2;r_1)|=0.
\end{align}
On the other hand, by Lemma \ref{lem:fixpoint-r} again, for $t\in[0,T]$,  we have
\begin{align}\label{continuity-3}
&|f(t;r_1)-f(t;r_2)|\nonumber\\
&\leq\Bigg|\int_{t}^T (f(s;r_1)-f(s;r_2))G(r_2,s,t)+(G(r_1,s,t)-G(r_2,s,t))f(s;r_1)ds\Bigg|\nonumber\\
&\quad+\left|H(r_1,z,t)-H(r_2,z,t)\right|\nonumber\\
&\leq\int_{t}^T |f(s;r_1)-f(s;r_2)|G(r_2,s,t)ds+\sup_{(s,t)\in[0,T]^2}|(G(r_1,s,t)-G(r_2,s,t))f(s;r_1)|T\nonumber\\
&\quad+\sup_{s\in[0,T]}\left|H(r_1,z,t)-H(r_2,z,t)\right|\nonumber\\
&\leq K_1\int_{t}^T |f(s;r_1)-f(s;r_2)|ds+K_2(r_1,r_2),
\end{align}
where the positive constant 
\begin{align*}
K_1:=\sup_{(r,t_2)\in {\cal M}_1^{(r_1,t_1)},~s\in[t,T]}G(r,s,t)<+\infty,
\end{align*}
 and the function $K_2(r_1,r_2)$ is given by
\begin{align*}
K_2(r_1,r_2):=&\sup_{(s,t)\in[0,T]^2}|(G(r_1,s,t)-G(r_2,s,t))f(s;r_1)|T\\
&+\sup_{s\in[0,T]}\left|H(r_1,z,s)-H(r_2,z,s)\right|,
\end{align*}
which satisfies $\lim_{\epsilon \to 0}K_2(r_1,r_2)\to 0$. Since \eqref{continuity-3} holds for any $t\in[0,T]$, replacing $t$ with $T-t$, we deduce that
\begin{align*}
&\left|f(T-t;r_1)-f(T-t;r_2)\right|\\
&\leq K_1\int_{T-t}^T |f(s;r_1)-f(s;r_2)|ds+K_2(r_1,r_2)\nonumber\\
&=K_1\int_{0}^t |f(T-s;r_1)-f(T-s;r_2)|ds+K_2(r_1,r_2).
\end{align*}
It follows from the Gronwall's inequality that
\begin{align}\label{continuity-4}
&\lim_{\epsilon \to 0}\left|f(T-t;r_1)-f(T-t;r_2)\right|\qquad\leq \lim_{\epsilon \to 0}K_1 e^{K_2(r_1,r_2)T}=0.
\end{align}
Taking $t=T-t_2$ in \eqref{continuity-4} and noting $\lim_{\epsilon \to 0}|f(t_2;r_1)-f(t_2;r_2)|=0$, together with \eqref{continuity} and \eqref{continuity-2}, we conclude that the function $(r,t)\mapsto f(t;r)$ is continuous. \hfill$\blacksquare$

\begin{proof}[Proof of Lemma \ref{lem:approximate}]
We first claim that $\lim_{n\to \infty}\mu^{i,n}\Ex[\theta_s^{*,i,n}]=f^*(s)$ for $s\in[t,T]$. To this purpose, we introduce the following auxiliary stochastic control problem given by, for $(t,x_0,z_0)\in[0,T]\times\R_+^2$,
\begin{align}\label{eq:auxiliary}
\tilde{u}^{i}(t,x_0,z_0):=\sup_{\theta\in \mathbb{U}_t^(x_0,z_0)}\Ex\left[  -\int_t^{T} e^{-\rho (s-t)}d\tilde{L}_s^{i}\big|\tilde{X}_t^{i}=x_0,\tilde{Z}_t^{i}=z_0 \right].
\end{align}
The value function \eqref{eq:auxiliary} is defined identically to \eqref{state-X-i} except that $(\mu^{i,n}, \sigma^{i,n}, \mu_Z^{i,n}, \sigma_Z^{i,n}, \lambda^{i,n})$ are replaced by $(\mu, \sigma, \mu_Z, \sigma_Z, \lambda)$. One can directly see that the optimal portfolio strategy $\tilde{\theta}^{*,i}=(\tilde{\theta}^{*,i}_s)_{s\in[t,T]}$ of problem \eqref{eq:auxiliary} admits the expression \eqref {optimal-portfolio-auxiliary} with $(\mu^{i,n}, \sigma^{i,n}, \mu_Z^{i,n}, \sigma_Z^{i,n}, \lambda^{i,n})$ being replaced by $(\mu, \sigma, \mu_Z, \sigma_Z, \lambda)$.

Using the assumption that $(\mu^{i,n}, \sigma^{i,n}, \mu_Z^{i,n}, \sigma_Z^{i,n}, \lambda^{i,n})$ converge uniformly to a common limit $(\mu, \sigma, \mu_Z, \sigma_Z, \lambda)\in(0,\infty)^4\times[0,1]$ as $n \to \infty$, it follows that, for any $s\in[t,T]$, $\theta_s^{*,i,n}$ converges to $\tilde{\theta}_s^{*,i}$ almost surely as $n\to\infty$. Since $\theta_s^{*,i,n}\in\mathbb{U}_t(x_0,z_0)$, the family $(\theta_s^{*,i,n})_{n\geq 1}$ is uniformly integrable. This  yields that, for any $s\in[t,T]$,
\begin{align*}
\lim_{n\to \infty}\mu^{i,n}\Ex\left[\theta_s^{*,i,n}\right]=\mu\Ex\left[\tilde{\theta}_s^{*,i}\right]=\mu\Ex\left[\theta_s^{*}\right]=f^*(s),
\end{align*}
where, the 2nd equality holds true because $\tilde{\theta}_s^{*,i}$ has the same distribution as $\theta^{*}$ given in Theorem 4.8. Now, by applying It{\^ o}'s formula to $|X_s^{*,i,n}-\bar{X}_s^{*,i,n}|^2$ and taking expectation, we obtain 
\begin{align}\label{eq:ineq-X}
\Ex\left[\left|X_s^{*,i,n}-\bar{X}_s^{*,i,n}\right|^2\right]&=2\Ex\left[\int_t^s  \left|X_\ell^{*,i,n}-\bar{X}_\ell^{*,i,n}\right|\left|\lambda^{i,n} f^*(\ell)- \frac{\lambda^{i,n}}{n}\sum_{j=1}^n \mu^{j,n}\theta^{*,j,n}_\ell\right|d\ell\right]\nonumber\\
&\leq \int_t^s  \Ex\left[\left|X_\ell^{*,i,n}-\bar{X}_\ell^{*,i,n}\right|^2\right]d\ell +\int_t^s \Ex\left[\left|f^*(\ell)- \frac{1}{n}\sum_{j=1}^n \mu^{j,n}\theta^{*,j,n}_\ell\right|^2\right]d\ell\nonumber\\
&:=\int_t^s  \Ex\left[\left|X_\ell^{*,i,n}-\bar{X}_\ell^{*,i,n}\right|^2\right]d\ell +\int_t^sF_n(\ell)d\ell.
\end{align}
Applying the Gronwall's lemma, we arrive at
\begin{align}\label{eq:ineq-X-2}
\Ex\left[\left|X_s^{*,i,n}-\bar{X}_s^{*,i,n}\right|^2\right]\leq 
e^{s-t}\int_t^sF_n(\ell)d\ell.
\end{align}
Note that, it holds that, for any $\ell\in[t,T]$,
\begin{align}\label{eq:F}
F_n(\ell)=&\Ex\left[\left|f^*(\ell)- \frac{1}{n}\sum_{j=1}^n \mu^{j,n}\theta^{*,j,n}_{\ell}\right|^2\right]\nonumber\\
&\leq 2 \Ex\left[\left|f^*(\ell)- \frac{1}{n}\sum_{j=1}^n \mu^{j,n}\Ex\left[\theta^{*,j,n}_\ell\right]\right|^2\right]+2 \Ex\left[\left|\frac{1}{n}\sum_{j=1}^n \mu^{j,n}\Ex\left[\theta^{*,j,n}_\ell\right]-\sum_{j=1}^n \mu^{j,n}\theta^{*,j,n}_\ell\right|^2\right]\nonumber\\
&\leq  \frac{2}{n^2}\sum_{j=1}^n\left|f^*(\ell)-  \mu^{j,n}\Ex[\theta^{*,j,n}_\ell]\right|^2+\frac{2}{n^2}\sum_{j=1}^n \Ex\left[\left| \mu^{j,n}\Ex\left[\theta^{*,j,n}_\ell\right]- \mu^{j,n}\theta^{*,j,n}_\ell\right|^2\right]\\
&\leq  \frac{2}{n^2}\sum_{j=1}^n\left|f^*(\ell)-  \mu^{j,n}\Ex\left[\theta^{*,j,n}_\ell\right]\right|^2+\frac{4}{n^2}\sum_{j=1}^n \left[ (\mu^{j,n})^2\Ex\left[\theta^{*,j,n}_\ell\right]^2+(\mu^{j,n})^2\Ex\left[|\theta^{*,j,n}_\ell|^2\right]\right].\nonumber
\end{align}
Using $\lim_{n\to\infty} \mu^{i,n}\mathbb{E}[\theta_s^{*,i,n}] = f^*(s)$ for $s\in[t,T]$, $\lim_{n\to\infty}\mu^{i,n} = \mu$, and the uniform integrability of $(\theta^{*,j,n})_{n\geq1}\subset\mathbb{U}_t(x_0,z_0)$, we deduce from \eqref{eq:F} that $\lim_{n\to\infty} F_n(\ell) = 0$ for any $\ell\in[t,T]$. From \eqref{eq:ineq-X-2} and Dominated Convergence Theorem (DCT), it follows that
\begin{align}\label{eq:ineq-X-3}
\sup_{s\in[t,T]}\Ex\left[\left|X_s^{*,i,n}-\bar{X}_s^{*,i,n}\right|^2\right]\leq 
&e^{T-t}\int_t^TF_n(\ell)d\ell\to 0,~~ \text{as}~n\to \infty.
\end{align}
Consequently, we have, as $n\to\infty$,
{\small\begin{align*}
\sup_{s\in[t,T]}\Ex\left[\left|L_s^{*,i,n}-\bar{L}_s^{*,i,n}\right|\right]&\leq \sup_{s\in[t,T]}\Ex\left[\left|X_s^{*,i,n}-\bar{X}_s^{*,i,n}\right|\right]+\int_t^T\Ex\left[\left|f^*(\ell)- \frac{\lambda^i}{n}\sum_{j=1}^n \mu^{j,n}\theta^{*,j,n}\ell\right|\right]d\ell\nonumber\\
&\leq \sup_{s\in[t,T]}\Ex\left[\left|X_s^{*,i,n}-\bar{X}_s^{*,i,n}\right|\right]+\int_t^T \sqrt{F_n(\ell)}d\ell.
\end{align*}}In a similar manner, we can establish \eqref{eq:approximate-2}, which completes the proof of the lemma.
\end{proof}

\noindent\textbf{Acknowledgements}  L. Bo and Y. Huang are supported by National Natural Science of Foundation China (No. 12471451), Natural Science Basic Research Program of Shaanxi (No. 2023-JC-JQ-05), Shaanxi Fundamental Science Research Project for Mathematics and Physics (No. 23JSZ010) and Fundamental Research Funds for the Central Universities (No. 20199235177). X. Yu is supported by the Hong Kong RGC General Research Fund (GRF) under grant no. 15306523 and by the Hong Kong Polytechnic University research grant under no. P0045654.

\end{document}